\documentclass[12pt]{article}
\usepackage{fullpage, comment, authblk, stackrel}
\usepackage[small]{caption}
\usepackage{tikz-cd}
\usepackage{subcaption}
\usepackage{amstext} 
\usepackage{array}  

\newcolumntype{L}{>{$}l<{$}} 
\newcolumntype{C}{>{$}c<{$}} 

\usepackage{nicefrac}
\usepackage{hyperref}
\usepackage{amssymb,mathtools,amsthm, amsmath}
\usepackage[all,cmtip]{xy}
\usepackage{eulervm,xfrac,mathdots}
\usepackage{rotating, setspace}

\newcommand{\blocktheorem}[1]{%
  \csletcs{old#1}{#1}
  \csletcs{endold#1}{end#1}
  \RenewDocumentEnvironment{#1}{o}
    {\par\addvspace{1.5ex}
     \noindent\begin{minipage}{\textwidth}
     \IfNoValueTF{##1}
       {\csuse{old#1}}
       {\csuse{old#1}[##1]}}
    {\csuse{endold#1}
     \end{minipage}
     \par\addvspace{1.5ex}}
}

\newcommand{\Rspace}        	{{\mathbb R}}

\newcommand{\Z}        		{{\mathbb Z}}

\newcommand{\mono}		{\hookrightarrow}

\newcommand{\Bfunc}          	{{\mathsf{B}}}
\newcommand{\Cfunc}          	{{\mathsf{C}}}

\newcommand{\Ffunc}          	{F}
\newcommand{\Gfunc}          	{G}
\newcommand{\Hfunc}          	{H}

\newcommand{\colim}		{\mathsf{colim}}

\newcommand{\Dist}			{{d}}

\newcommand{\field}			{\mathsf{k}}

\newcommand{\ee}			{\varepsilon}

\newcommand{\op}			{\mathsf{op}}

\newcommand\define[1]		{{\bf{#1}}}

\newcommand{\Int}          		{{\mathbb{I}}}

\makeatletter
\def\moverlay{\mathpalette\mov@rlay}
\def\mov@rlay#1#2{\leavevmode\vtop{%
   \baselineskip\z@skip \lineskiplimit-\maxdimen
   \ialign{\hfil$\m@th#1##$\hfil\cr#2\crcr}}}
\newcommand{\charfusion}[3][\mathord]{
    #1{\ifx#1\mathop\vphantom{#2}\fi
        \mathpalette\mov@rlay{#2\cr#3}
      }
    \ifx#1\mathop\expandafter\displaylimits\fi}
\makeatother

\newtheoremstyle{amit}
{7pt}
{7pt}
{}
{7pt}
{\bf}
{:}
{.5em}
{}

\hypersetup{
  colorlinks   = true, 
  urlcolor     = blue, 
  linkcolor    = blue, 
  citecolor   = blue 
}

\newcommand{\diam}			{{\text{diam}}}

\newcommand{\Mon}		{{\mathsf{Mon}}}
\newcommand{\Lan}			{{\mathsf{Lan}}}
\newcommand{\Fnc}			{{\mathsf{Fnc}}}
\newcommand{\Fil}			{{\mathsf{Fil}}}

\newcommand{\BD}			{{{\mathsf{ZB}}}}
\newcommand{\M}			{{{\mathsf{MI}}}}
\newcommand{\Ch}			{{{\mathsf{Ch}}}}
\newcommand{\Zfunc}		{{{\mathsf{Z}}}}
\renewcommand{\Vec}		{{{\mathsf{Vec}}}}

\newcommand{\rank}		{{{\mathsf{rk\,}}}}

\theoremstyle{amit}
\newtheorem{defn}{Definition}[section]
\newtheorem{prop}[defn]{Proposition}
\newtheorem{lem}[defn]{Lemma}
\newtheorem{thm}[defn]{Theorem}

\newtheorem{rmk}[defn]{Remark}
\newtheorem{ex}[defn]{Example}

\begin{document}


\title{Edit Distance and Persistence Diagrams Over Lattices
\thanks {This material is based upon work supported by the National Science Foundation under 
Grant No.\ 1717159.}}
\author[1]{Alexander McCleary}
\author[1]{Amit Patel}
\affil[1]{Department of Mathematics, Colorado State University}
\date{}

\maketitle

\begin{abstract}
We build a functorial pipeline for persistent homology. 
The input to this pipeline
is a filtered simplicial complex indexed by any finite metric lattice and the output is a persistence
diagram defined as the M\"obius inversion of its birth-death function.
We adapt the
Reeb graph edit distance to each of our categories
and prove that both functors in our pipeline are $1$-Lipschitz making our pipeline stable.
Our constructions generalize the classical persistence diagram and, in this setting,
the bottleneck distance is strongly equivalent to the edit distance.
\end{abstract}

\section{Introduction}
\label{sec:introduction}

In the most basic setting, persistent homology takes as input a
finite $1$-parameter filtration 
$K_0 \subseteq K_1 \subseteq \cdots \subseteq K_n = K$
of a finite simplicial complex $K$ and outputs a persistence diagram.
The persistence diagram, as originally defined in \cite{CSEdH} and equivalently in \cite{size_theory, 10.1117/12.279674}, 
is roughly defined as follows.
Fix a field $\field$.
For every pair of indices $a \leq b$, let $f_\ast [a,b]$ 
be the rank of the $\field$-linear map on homology $\Hfunc_\ast ( K_a; \field ) \to \Hfunc_\ast (K_b ; \field)$ 
induced by the inclusion $K_a \subseteq K_b$.
The persistence diagram $g_\ast$ is the assignment to every pair $a \leq b$
the following signed sum:
	\begin{equation}
	\label{eq:inclusion_exclusion}
	g_\ast [a,b] := f_\ast[a, b-1] - f_\ast[a, b] + f_\ast[a-1, b] - f_\ast[a-1, b-1].
	\end{equation}
We interpret the integer $g_\ast[a,b]$ as the number of independent
cycles that appear at $a$ and become boundaries at $b$.
The most important property of the persistence diagram is that it is stable
to perturbations of the filtration on $K$ \cite{CSEdH}.
See \S \ref{sec:applications} for a discussion of stability and applications.


In \cite{Patel2018}, the second author identifies Equation (\ref{eq:inclusion_exclusion}) as a special case
of the M\"obius inversion formula.
This observation allows for generalizations of the persistence diagram in at least two directions.
First, we may consider any coefficient ring and still get
a well defined persistence diagram.
In fact, any constructible $1$-parameter persistence module valued in any essentially small
abelian category admits a persistence diagram.
Furthermore, the bottleneck stability theorem of \cite{CSEdH} generalizes to this setting~\cite{bottleneck}.
Second, we may consider as input a filtration indexed over any locally finite poset
and still get a well defined persistence diagram as demonstrated by \cite{kim2018generalized}.
However, it is only now, in this paper, that we are able to establish a statement of stability for these
more general persistence diagrams thus opening a path to applications; see~\S \ref{sec:applications}.

The persistence diagram is just one invariant of a filtered space.
For example, the persistence landscape \cite{landscapes} is a second invariant that is better suited 
for machine learning and statistical methods \cite{landscapes_properties, 7299106}.
The work of \cite{betthauser2019graded} recasts persistence landscapes as a M\"obius inversion
and, in the process, uncovers previously unknown structure.

\paragraph{Contributions}
We establish and study the following pipeline of functors for persistent homology with coefficients
in a fixed field $\field$:
	\begin{equation*}
	\begin{tikzcd}
	\Fil(K) \ar[rr, "\BD_\ast"] && \Mon \ar[rr, "\M"] && \Fnc.
	\end{tikzcd}
	\end{equation*}
The input to the pipeline is the category of filtrations $\Fil(K)$.
Its objects are filtrations of finite simplicial complexes indexed over finite metric lattices.
The second category $\Mon$ consists of monotone integral functions over finite metric lattices.
To every object in $\Fil(K)$, the functor $\BD_\ast$ assigns a monotone integral function.
Instead of using the rank function of a filtration, as mentioned earlier, we use the birth-death 
function $\BD_\ast$ of a filtration.
The birth-death function is inspired by the work of Henselman-Petrusek and Ghrist \cite{GregH1, GregH2}.
In a way, the two functions are equivalent; see \S \ref{sec:ordinary}.
The third category $\Fnc$ consists of integral, but not necessarily monotone, functions over finite metric lattices.
To every object in $\Mon$, the M\"obius inversion functor~$\M$ assigns
its M\"obius inversion, which is an integral function but not necessarily monotone.
Morphisms in the first two categories are inspired by the definition of the
\emph{Reeb graph elementary edit} of Di Fabio and Landi \cite{Di_Fabio}.
We think of the morphisms in $\Fil(K)$ as a deformation of one filtration to another.
This deformation is formally a Kan extension.
Morphisms in $\Mon$ are defined in the same way, but morphisms in $\Fnc$ 
are inspired by morphisms between signed measure spaces.
Our main theorem, Theorem \ref{thm:main}, says that both functors
in the pipeline are $1$-Lipschitz.
The three metrics, one for each category, are all inspired
by the categorification
of the \emph{Reeb graph edit distance} by Bauer, Landi, and M\'emoli~\cite{edit_distance_universal}.
Hence, we call all three metrics the \emph{edit distance}.
Finally in \S \ref{sec:ordinary}, we prove that our definition of the persistence diagram
generalizes the classical definition \cite{CSEdH, size_theory,10.1117/12.279674} and that, in this setting, the edit distance between persistence
diagrams is strongly equivalent to the bottleneck distance; see Theorem \ref{thm:bottleneckequivalence}.
%

\paragraph{Outline}
We start with a discussion of stability and its importance in applications
of persistent homology.
In \S \ref{sec:one}, we establish basic definitions and properties of metric lattices.
The next three sections \S \ref{sec:two}, \S \ref{sec:three}, and \S \ref{sec:four}
establish the three categories $\Fil(K)$, $\Mon$, and $\Fnc$, respectively, along 
with the functors $\BD_\ast$ and $\M$.
In \S \ref{sec:five}, we define the edit distance in all three categories and prove
that both functors $\BD_\ast$ and $\M$ are $1$-Lipschitz.
We put the pieces together in \S \ref{sec:persistence_diagrams} and state our main theorem.
Finally in \S \ref{sec:ordinary}, we verify that our framework generalizes the classical persistence diagram,
and we prove that the bottleneck distance is strongly equivalent to the edit distance.

\section{Applications of Stability}
\label{sec:applications}

Scientists study natural phenomena by collecting and analyzing data.
Data is a finite set usually with additional structure such as, for example, a metric
or an embedding into a vector space.
The goal is to extract information and then use this information to build models 
or theories.
However, measurements are noisy and so any information
that is extracted must be stable to noise.
Persistent homology is an attractive tool for studying the shape of data precisely
because it extracts stable information in the form of a persistence diagram.

A major drawback of classical persistent homology is its sensitivity
to outliers. 
An outlier is, roughly speaking, a data point that is too far away from where it should be.
The persistence diagram changes drastically by the introduction of just one outlier.
This is not desirable as outliers are common in many data sets.
The solution requires a theory of persistent homology that can handle a
two-parameter filtration of a space.
Such a theory not only requires a generalization of the classical persistence diagram but also
a statement of stability for the reasons described above.
In this paper, we present a generalization of the persistence diagram for 
simplicial complexes filtered over any finite lattice, but, as mentioned in the last section,
this is not entirely new.
One of the achievements of this paper is a first ever statement of stability 
for this more general setting; see Theorem \ref{thm:main}.
To our surprise, our stability theorem closely resembles the
bottleneck stability theorem; see Theorem \ref{thm:bottleneckequivalence}.

Functoriality of persistent homology is the second main achievement of this paper.
Considering that category theory has its roots in algebraic topology, it may be surprising that
the assignment of a persistence diagram to a filtration was, until now, not known to be functorial.
Since its inception in the 1940's, category theory has permeated every field of mathematics, logic, and
parts of computer science.
It is only reasonable to expect that functoriality will play a major role in the future development
of applications for persistent homology.
For example, the 1-parameter family of persistence diagrams associated to 
a time-varying data set can now be described
as a constructible cosheaf \cite{curry_patel} of persistence diagrams.

We now review a few applications of classical persistent homology that are
well known within the applied topology community.
All rely on stability.

\paragraph{Homological Inference}
One of the first applications of the bottleneck stability theorem is homological inference, which appears
in the same paper as the theorem~\cite{CSEdH}.
Data often lives along a lower dimensional subspace of a higher dimensional vector space.
In this case, it is useful
to assume that the data is sampled from some sufficiently nice subspace, say $X \mono \Rspace^n$.
For example, $X$ could be a smooth manifold, a compact Whitney stratified space, or
a piecewise-linear embedding of a finite simplicial complex.
A natural question to ask is the following: How can one infer the homology of $X$ 
from a finite sample $Y \subseteq X$? 
The answer to this question lies in a quantity called the \emph{homological feature size} of $X$.
The statement is roughly as follows.
Suppose $Y$ is chosen so that the Hausdorff distance between $X$ and $Y$ in $\Rspace^n$ is at
most $r/4$, where $r > 0$ is the homological feature size of $X$.
Then the homology of $X$ can be read from the classical persistence diagrams associated to $Y$.

The above sampling condition relies on the distance between $X$ and the sample $Y$,
which is highly sensitive to outliers.
One way of reducing the impact of outliers is by thinking of $X$ and $Y$ as measures and then using
the Wasserstein distance between the two measures~\cite{distance_to_measure}.
One can then imagine setting up a two-filtration: the first parameter filters by distance, as in the case above,
and the second parameter filters by mass.
We suspect Theorem \ref{thm:main} implies a homological inference theorem for this two-parameter setting.

\paragraph{Machine Learning}
The growing field of data science offers a wide variety of tools for extracting
information from data.
However, most of these tools require as input a vector, but a persistence diagram 
is far from a vector.
In order to make use of these tools, one must vectorize the persistence diagram.
\emph{Persistence images} is one popular way of vectorizing the classical persistence diagram \cite{JMLR:v18:16-337}.
Again, because data is inherently noisy, it is crucial that any vectorization is stable to noise.
Persistence images are stable.

In \cite{10.5555/3294771.3294927}, the authors develop an input layer for deep neural networks that
takes a classical persistence diagram and computes a parametrized projection 
that can be learned during network training. 
This layer is designed in a way that is stable to perturbations of the input persistence diagrams.

Our main theorem, Theorem \ref{thm:main}, lays the foundation for using our M\"obius inversion
based, multi-parameter persistence diagrams for applications in machine learning.

\section{Preliminaries}
\label{sec:one}
We start with an introduction to bounded lattices and bounded lattice functions.
From here, we equip our lattices with a metric and discuss the distortion
of a lattice function between two metric lattices.

\subsection{Lattices}
\label{sec:one.lattices}

A \emph{poset} is a set $P$ with a reflexive, antisymmetric, and transitive relation $\leq$.
For two elements $a, b \in P$ in a poset, we write $a < b$ to mean $a \leq b$ and $a \neq b$.
For any $a \leq c$, the \emph{interval} $[a,c] \subseteq P$ is the subposet consisting of all $b \in P$ such
that $a \leq b \leq c$.
The poset $P$ has a \emph{bottom} if there is an element $\bot \in P$ such that
$\bot \leq a$ for all $a \in P$.
The poset $P$ has a \emph{top} if there is an element $\top \in P$ such that
$a \leq \top$ for all $a \in P$.
A function $\alpha : P \to Q$ between two posets is \emph{monotone} if for all $a \leq b$, $\alpha(a) \leq \alpha(b)$.

The \emph{meet} of two elements $a,b \in P$ in a poset, written $a \wedge b$, is the greatest lower bound
of $a$ and $b$.
The \emph{join} of two elements $a,b \in P$ in a poset, written $a \vee b$, is the least upper bound of 
$a$ and $b$.
The poset $P$ is a \emph{lattice} if both joins and meets exist for all pairs of elements in $P$.
A lattice is \emph{bounded} if it contains both a top and a bottom.
If $P$ is a finite lattice, then the existence of meets and joins implies that $P$ has a bottom and a top,
respectively.
Therefore all finite lattices are bounded.
A function $\alpha : P \to Q$ between two bounded lattices is a \emph{bounded lattice function} if $\alpha(\top) = \top$,
$\alpha(\bot) = \bot$, and for all $a,b \in P$,
$\alpha(a \vee b) = \alpha(a) \vee \alpha(b)$ and $\alpha(a \wedge b) = \alpha(a) \wedge \alpha(b)$.
Note that bounded lattice functions are monotone.
This is because $a \leq b$ if and only if $a \wedge b = a$, and therefore
$$\alpha(a) = \alpha(a \wedge b) = \alpha(a) \wedge \alpha(b) \Longrightarrow \alpha(a) \leq \alpha(b).$$

\begin{prop}
\label{prop:maximal}
Let $P$ and $Q$ be finite lattices and $\alpha : P \to Q$ a bounded lattice function. 
Then for all $a \in Q$, the pre-image $\alpha^{-1} [\bot, a] $ has a maximal element.
\end{prop}
\begin{proof}
The pre-image is non-empty because $\alpha(\bot) = \bot$.
The pre-image is finite because $P$ is finite.
For any two elements $b$ and $c$ in the pre-image, both $b \vee c$ and 
$b \wedge c$ are also in the pre-image because
	\begin{align*}
	\alpha(b \vee c) = \alpha(b) \vee \alpha(c) \leq a \vee a = a	&&
	\alpha(b \wedge c) = \alpha(b) \wedge \alpha(c) \leq a \wedge a = a.
	\end{align*}
Thus $\alpha^{-1}[\bot, a]$ is a finite lattice and all finite lattices
have a unique maximal element.
\end{proof}

For a finite lattice $P$, let $\Int P  := \big\{ [a,b] \subseteq P : a \leq b \big \}$ be its set of intervals.
The product order on $P \times P$ restricts to a partial order $\preceq$ on $\Int P$ as follows:
$[a, b] \preceq [c, d]$ if $a \leq c$ and $b \leq d$.
The join of two intervals is $[a,b] \vee [c,d] = [ a \vee c, b \vee d]$,
and the meet of two intervals is $[a,b] \wedge [c,d] = [a \wedge c, b \wedge d]$.
All this makes $\Int P$ a finite lattice.
Its bottom element is $[\bot, \bot]$ and its top element is $[\top, \top]$.

A bounded lattice function $\alpha : P \to Q$ between two finite lattices induces a bounded lattice
function $\Int \alpha : \Int P \to \Int Q$ as follows.
For an interval $[a,b] \in \Int P$, let $\Int \alpha \big( [a,b] \big) := \big[ \alpha(a), \alpha(b) \big]$.
We have
	\begin{align*}	
	\Int \alpha \big( [a,b] \wedge [c,d] \big) &= \Int \alpha \big( [a \wedge c, b \wedge d] \big) = 
	\big[ \alpha(a \wedge c) , \alpha (b \wedge d) \big] \\
	 &= 
	\big[ \alpha(a) \wedge \alpha(c) , \alpha(b) \wedge \alpha(d) \big]
	= \Int \alpha \big( [a,b] \big)  \wedge \Int \alpha \big( [c,d] \big). \\
	\Int \alpha \big( [a,b] \vee [c,d] \big) &= \Int \alpha \big( [a \vee c, b \vee d] \big) = 
	\big[ \alpha(a \vee c) , \alpha (b \vee d) \big] \\
	 &= 
	\big[ \alpha(a) \vee \alpha(c) , \alpha(b) \vee \alpha(d) \big]
	= \Int \alpha \big( [a,b] \big)  \vee \Int \alpha \big( [c,d] \big) \\
	\Int \alpha \big( [\bot, \bot] \big) &= \big[ \alpha(\bot), \alpha(\bot) \big] = [\bot, \bot] \\
	\Int \alpha \big( [\top, \top] \big) &= \big[ \alpha(\top), \alpha(\top) \big] = [\top, \top].
	\end{align*}
Thus $\Int \alpha$ is a bounded lattice function.
Further, for any pair of bounded lattice functions $\alpha : P \to Q$ and $\beta : Q \to R$,
$\Int (\beta \circ \alpha) = \Int \beta \circ \Int \alpha$.
All this makes $\Int$ an endofunctor on the category of finite lattices and bounded lattice functions.
To minimize notation, we will write~$\bar P$ for $\Int P$ and $\bar \alpha$ for $\Int \alpha$.

\subsection{Metric Lattices}
\label{sec:one.metric}

A \emph{finite (extended) metric lattice} is a tuple $(P, d_P)$ where $P$ is a finite lattice and
$d_P : P \times P \to \Rspace^{\geq 0} \cup \{\infty\}$ an extended metric.
A \emph{morphism of finite metric lattices} $\alpha : (P, d_P) \to (Q, d_Q)$
is a bounded lattice function $\alpha : P \to Q$.
The \emph{distortion} (see \cite[Definition 7.1.4]{burago2001course}) of a morphism $\alpha : (P, d_P) \to (Q, d_Q)$, 
denoted $|| \alpha ||$, is 
	$$ || \alpha || := \max_{a,b \in P} \big | d_P (a,b) - d_Q \big( \alpha(a), \alpha(b) \big) \big|.$$
Note that $|| \alpha ||$ might be infinite because the distance between any two points might 
be infinite.
To minimize notation, we will write finite metric lattices $(P, d_P)$ simply as $P$
with the implied metric $d_P$. 

\begin{ex}
\label{ex:one}
See Figure \ref{fig:lattice_map} for two examples of finite metric lattices $P$ and $Q$
and a morphism of finite metric lattices $\alpha : P \to Q$.
The distortion of $\alpha$ is $|| \alpha || = 1$.
Forthcoming examples will build on this one example.
	\begin{figure}
	\centering 
	\includegraphics{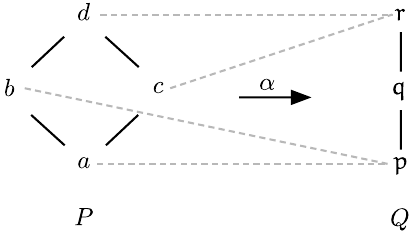}
	\caption{Hasse diagrams of two finite metric lattices $P$ and $Q$.
	The metrics $d_P$ and $d_Q$ assigns to every pair of elements the length (i.e.\ the number of
	edges) of the shortest path between them.
	For example, $d_P(a,d) = 2$ and $d_Q(p, q) = 1$.
	The function $\alpha : P \to Q$ defined as $\alpha(a) = \alpha(b) = p$
	and $\alpha(c) = \alpha(d) = r$ is a bounded lattice function.
	The distortion of $\alpha$ is $|| \alpha || = 1$.}
	\label{fig:lattice_map}
	\end{figure}
\end{ex}

For every finite metric lattice $P$, we have the finite metric lattice of intervals
$\bar P$ where 
$d_{\bar P}\big([a,b], [c,d] \big) := \max \big \{ d_P(a,c), d_P (b,d) \big \}.$
A morphism $\alpha : P \to Q$ of finite metric lattices induces a morphism
of finite metric lattices $\bar \alpha : \bar P  \to \bar Q$.
The distortion of $\bar \alpha$ is
$$ || \bar \alpha || := \max_{ [a,b], [c,d] \in \bar P} \big| \max \big\{ d_P(a,c), d_P(b,d) \big\} - 
\max \big\{ d_Q (\alpha(a), \alpha(c)), d_Q (\alpha(b), \alpha(d)) \big\} \big |.$$
Proposition \ref{prop:distortion} says that the two distortions $|| \alpha ||$ and $|| \bar \alpha ||$
are equal. Its proof requires the following lemma.

\begin{lem}{\cite[Lemma 3 page 31]{hierarchical_clustering}}
\label{lem:inequality}
For all non-negative real numbers $w, x, y, z \in \Rspace^{\geq 0}$, 
	$$ \big | \max(w,x) - \max(y,z) \big | \leq \max \big( |w-y|, |x-z| \big).$$
\end{lem}

\begin{prop}
\label{prop:distortion}
Let $\alpha : P \to Q$ be a bounded lattice function between two finite metric lattices and let
$\bar \alpha : \bar P \to \bar Q$ be the induced bounded lattice function on intervals.
Then $|| \bar \alpha || = || \alpha ||$. 
\end{prop}
\begin{proof}
First we show $|| \bar \alpha || \geq || \alpha ||$.
If $|| \alpha || = \ee$, then there are elements $a,b\in P$ such that
$\ee = \big | d_P(a,b) - d_Q (\alpha(a), \alpha(b)) \big|$.
For the intervals $[a,a]$ and $[b,b]$, we have
$$\big | \max \big \{ d_P(a,b), d_P (a,b) \big \} - \max \big \{ d_Q(\alpha(a), \alpha(b)), 
d_Q (\alpha(a), \alpha(b)) \big \} \big | = \ee$$
proving the claim.
Now we show $|| \bar \alpha || \leq || \alpha ||$ using Lemma \ref{lem:inequality}:
	\begin{align*}
	|| \bar \alpha || & := \max_{ [a,b], [c,d] \in \bar P} \big| \max \big\{ d_P(a,c), d_P(b,d) \big\} - 
\max \big\{ d_Q (\alpha(a), \alpha(c)), d_Q (\alpha(b), \alpha(d)) \big\} \big | \\
	& \leq \max_{ [a,b], [c,d] \in \bar P} \big \{ \big | d_P(a,c) - d_Q \big (\alpha(a), \alpha(c) \big) \big | , 
	\big | d_P (b,d) - d_Q\big(\alpha(b), \alpha(d)\big) \big|  \\
	& = || \alpha ||.
	\end{align*}
\end{proof}

\begin{ex}
\label{ex:three}
The morphism of finite metric lattices $\alpha : P \to Q$ in Example~\ref{ex:one} 
induces the morphism of finite metric lattices $\bar \alpha : \bar P \to  \bar Q$ 
in Figure \ref{fig:int_map}.
The distortion of $\bar \alpha$ is $|| \bar \alpha || = 1$.
	\begin{figure}
	\centering
	\includegraphics{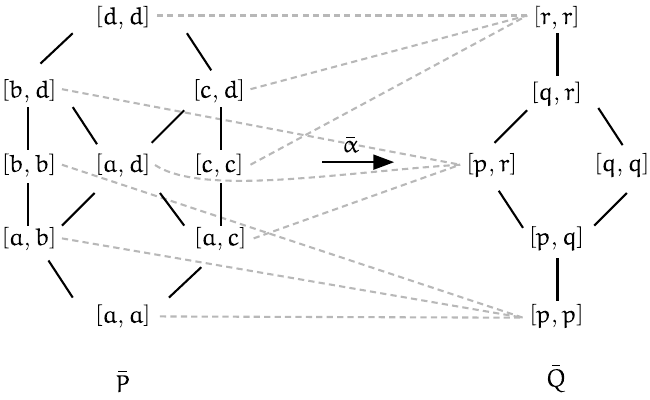}
	\caption{Hasse diagrams of the lattices $\bar P$ 
	and $\bar Q$ where
	$P$ and $Q$ are from Example~\ref{ex:one}.
	The morphism $\alpha : P \to Q$ from the same example
	extends to the morphism $\bar \alpha : \bar P \to \bar Q$
	as follows.
	The function $\bar \alpha$ sends $\big\{ [a,a], [a,b], [b,b] \big \}$ to $\big\{ [p,p] \big \}$, 
	$\big\{ [a,c], [a,d], [b,d] \big \}$
	to $\big \{ [p,r] \big \}$, and $\big \{ [c,c], [c,d], [d,d] \big \}$ to $\big \{ [r,r] \big \}$.
	The distortion of $\bar \alpha$ is $|| \bar \alpha || = || \alpha || = 1$.}
	\label{fig:int_map}
	\end{figure}
\end{ex}

\section{Filtrations}
\label{sec:two}
We now consider filtrations of a fixed finite simplicial complex indexed by finite metric lattices.
Fix a finite simplicial complex $K$ and denote by $\Delta K$ the category 
consisting of all subcomplexes $A \subseteq K$ as its objects and inclusions $A \mono B$ as morphisms.

\begin{defn}
Let $P$ be a finite metric lattice and $K$ a finite simplicial complex.
A filtration of $K$ indexed by $P$, or simply a 
\define{$P$-filtration of $K$},
is a functor $\Ffunc : P \to \Delta K$.
That is, for all $a \in P$, $\Ffunc(a)$ is a subcomplex of $K$ and for all $a \leq b$, $\Ffunc(a \leq b)$
is the inclusion of $\Ffunc(a)$ into $\Ffunc(b)$.
Further, we require that $\Ffunc(\top) = K$.
\end{defn}

\begin{defn}
A \define{filtration-preserving morphism} is a triple $(\Ffunc, \Gfunc, \alpha)$
where $\Ffunc : P \to \Delta K$ and $\Gfunc : Q \to \Delta K$ are $P$ and $Q$-filtrations of
$K$, respectively, and $\alpha : P \to Q$
is a bounded lattice function satisfying the following axiom.
For all $a \in Q$, $\Gfunc(a) = \Ffunc(a^\star)$ where $a^\star := \max \alpha^{-1}[\bot, a]$:
	\begin{equation*}
	\begin{tikzcd}
	P \ar[rr, "\alpha" ""{name=U, below}] \ar[rd,  "\Ffunc"] && Q \ar[ld, "\Gfunc"] \\
	& \Delta K .  & 
	\end{tikzcd}
	\end{equation*}
\end{defn}

\begin{rmk}
A more sophisticated but an equivalent definition of a filtration-preserving morphism is the following.
A filtration-preserving morphism is a triple $(\Ffunc, \Gfunc, \alpha)$
where $\Ffunc : P \to \Delta K$ and $\Gfunc : Q \to \Delta K$ are $P$ and $Q$-filtrations of
$K$, respectively, and $\alpha : P \to Q$
is a bounded lattice function such that
$G$ is the left Kan extension of $\Ffunc$ along $\alpha$, written $\Gfunc = \Lan_\alpha \Ffunc$:
	\begin{equation*}
	\begin{tikzcd}
	P \ar[rr, "\Ffunc" ""{name=U, below}] \ar[rd,  "\alpha"' ] && \Delta K \\
	& Q. \arrow[Rightarrow, from=U, "\mu"] \ar[ru, dashrightarrow, "\Gfunc = \Lan_\alpha \Ffunc "'] & 
	\end{tikzcd}
	\end{equation*}
By construction of the left Kan extension,
$$\Lan_\alpha \Ffunc (a) := \colim_{\Delta K} \Ffunc |_{\alpha^{-1} [\bot,a] }$$
for all $a \in Q$.
By Proposition \ref{prop:maximal}, $\alpha^{-1} [\bot, a]$ has a maximal element
$a^\star$ and therefore $\Lan_\alpha \Ffunc(a)$ is equal to $\Ffunc(a^\star)$.
For all $a \leq b$ in $Q$, $a^\star \leq b^\star$ inducing the inclusion $\Lan_\alpha \Ffunc(a \leq b)$.
The natural transformation $\mu : \Ffunc \Rightarrow \Gfunc \circ \alpha$ is gotten as follows.
For $c \in P$, let $a := \alpha(c)$.
Since $c \leq a^\star$ and $\Lan_\alpha \Ffunc(a)$ is equal to $\Ffunc(a^\star)$, 
we get the inclusion $\mu(c) : \Ffunc(c) \mono \Gfunc \circ \alpha(c) = \Gfunc(a)$.
\end{rmk}

\begin{rmk}
A zigzag of filtration-preserving morphisms categorifies the notion of a \emph{transposition}
introduced in \cite{vineyards}. 
Consider
two filtration-preserving morphisms $(\Ffunc, \Gfunc, \alpha)$ and $(\Hfunc, \Gfunc, \beta)$:
	\begin{equation*}
	\begin{tikzcd}
	P \ar[r, "\alpha" ""{name=U, below}] \ar[rd,  "\Ffunc"] & Q \ar[d, "\Gfunc"] & R \ar[ld, "\Hfunc"] \ar[l, "\beta"'] \\
	& \Delta K .  & 
	\end{tikzcd}
	\end{equation*}
Suppose for $q \in Q$, both $\alpha^{-1}(q)$ and $\beta^{-1}(q)$ are nonempty.
Then the simplices in $K$ that appear in the filtration $\Ffunc$ restricted to $\alpha^{-1}(q)$
appear at once in $\Gfunc$ at $q$.
Further, the same simplices that appear in $\Ffunc$ restricted
to $\alpha^{-1}(q)$ appear in $\Hfunc$ restricted to $\beta^{-1}(q)$ albeit in a 
possibly different order.
The two morphisms 
$(\Ffunc, \Gfunc, \alpha)$ and $(\Hfunc, \Gfunc, \beta)$ are together a generalization
of the notion of a \emph{transposition}.
\end{rmk}

\begin{ex}
\label{ex:two}
Let $\alpha : P \to Q$ be the bounded lattice function described in Example \ref{ex:one}.
Consider the two filtrations $\Ffunc : P \to \Delta K$ and $\Gfunc : Q \to \Delta K$ of the $2$-simplex $K$
in Figure \ref{fig:fil_map}.
The triple $(\Ffunc, \Gfunc, \alpha)$ is a filtration-preserving morphism $\alpha : \Ffunc \to \Gfunc$.
	\begin{figure}
	\centering \includegraphics{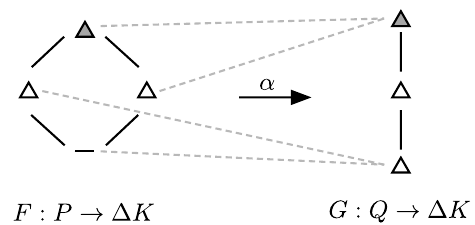}
	\caption{Filtrations $\Ffunc$ and $\Gfunc$ of the $2$-simplex along with
	a filtration-preserving morphism $\alpha$ as described in Example \ref{ex:one}.}
	\label{fig:fil_map}
	\end{figure}
\end{ex}

\begin{prop}
\label{prop:filtration_composition}
If $(\Ffunc, \Gfunc, \alpha)$ and $(\Gfunc, \Hfunc, \beta)$ 
are filtration-preserving morphisms, 
then $(\Ffunc, \Hfunc, \beta \circ \alpha)$ is a filtration-preserving morphism.
\end{prop}
\begin{proof}
Suppose $\Ffunc : P \to \Delta K$, $\Gfunc : Q \to \Delta K$, 
and $\Hfunc : R \to \Delta K$.
For all $a \in R$, $\Hfunc(a) = \Gfunc(a^\star)$ where $a^\star := \max \beta^{-1}[\bot, a]$.
Furthermore, $\Gfunc(a^\star) = \Ffunc(a^{\star \star})$ where $a^{\star \star} := \max \alpha^{-1}[\bot, a^\star]$.
Since $a^{\star \star} = \max (\beta \circ \alpha)^{-1} [\bot, a]$, we have that
$\Hfunc(a) = \Ffunc(a^{\star \star})$.
Thus $(\Ffunc, \Hfunc, \beta \circ \alpha)$ is a filtration-preserving morphism.
\end{proof}

\begin{defn}
Fix a finite simplicial complex $K$. 
Let $\Fil(K)$ be the category whose objects are $P$-filtrations of $K$, 
over all finite metric lattices~$P$, and whose morphisms are filtration-preserving morphisms.
We call $\Fil(K)$ the \define{category of filtrations of $K$}.
\end{defn}

There are ways to relate two filtration categories.
A simplicial map $f : K \to L$ induces a push-forward functor $f_\ast : \Fil(K) \to \Fil(L)$
and a pull-back functor $f^{\ast} : \Fil(L) \to \Fil(K)$.
Unfortunately, we do not need these functors.

\section{Monotone Integral Functions}
\label{sec:three}

We now define the category of monotone integral functions over finite metric lattices $\Mon$
and construct the birth-death functor $\BD_\ast : \Fil(K) \to \Mon$.
Let $\Z$ be the poset of integers with the usual total ordering $\leq$.

\begin{defn}
Let $P$ and $Q$ be two finite metric lattices and let $f : \bar P \to \Z$ and $g :\bar  Q \to \Z$
be two monotone integral functions on their lattice of intervals.
A \define{monotone-preserving morphism} from $f$ to $g$ is a triple $(f, g, \bar \alpha)$ where 
$f : \bar P \to \Z$ and $g : \bar Q \to \Z$ are monotone functions and 
$\bar \alpha : \bar P \to \bar Q$ is a bounded lattice function induced by a bounded
lattice function $\alpha : P \to Q$ satisfying the following axiom.
For all $I \in \bar Q$ and $I^\star := \max \bar \alpha^{-1} [\bot, I]$, $g(I) = f(I^\star)$:
	\begin{equation*}
	\begin{tikzcd}
	\bar P \ar[rd, "f" ] \ar[rr, " \bar \alpha"] && \bar Q \ar[ld, "g"] \\
	& \Z.   & 
	\end{tikzcd}
	\end{equation*}
\end{defn}

Note that if $(f, g, \bar \alpha)$ is a monotone-preserving morphism, then
$f [\top, \top] = g [\top, \top]$.

\begin{rmk}
A more sophisticated but an equivalent definition of a monotone-preserving morphism is the following.
A \define{monotone-preserving morphism} is a triple $(f, g, \bar \alpha)$ where 
$f : \bar P \to \Z$ and $g : \bar Q \to \Z$ are monotone functions and 
$\bar \alpha : \bar P \to \bar Q$ is a bounded lattice function induced by a bounded lattice
function $\alpha : P \to Q$ such that
$g$ is the left Kan extension of $f$ along $ \alpha$, written~$g = \Lan_{ \alpha} f$:
	\begin{equation*}
	\begin{tikzcd}
	\bar P \ar[rr, "f" ""{name=U, below}] \ar[rd, " \bar \alpha"' ] && \Z \\
	& \bar Q. \arrow[Rightarrow, from=U, "\mu"] \ar[ru, dashrightarrow, "g = \Lan_{ \bar \alpha} f"'] & 
	\end{tikzcd}
	\end{equation*}
\end{rmk}

\begin{ex}
\label{ex:four}
See Figure \ref{fig:mon_map} for an example of monotone integral functions $f$ and $g$ 
on the lattices $\bar P$ and $\bar Q$, respectively, from Example \ref{ex:three}. 
The triple $(f, g, \bar \alpha)$, where $\bar \alpha : \bar P \to \bar Q$ is from the same example, 
is a monotone-preserving morphism.
	\begin{figure}
	\centering
	\includegraphics{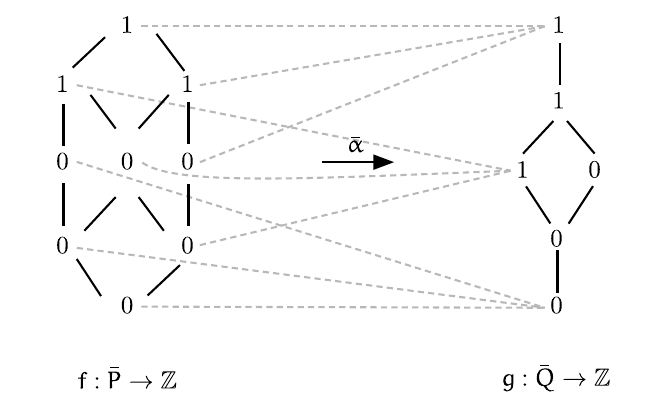}
	\caption{Two monotone integral functions $f$ and $g$ on the metric lattices 
	$\bar P$ and $\bar Q$ from Example~\ref{ex:three}.
	The triple $(f, g, \bar \alpha)$, where $\bar \alpha : \bar P \to \bar Q$ is from the same example,
	is a monotone-preserving morphism from $f$ to $g$.}
	\label{fig:mon_map}
	\end{figure}
\end{ex}

\begin{prop}
\label{prop:monotone_composition}
If $(f,g, \bar \alpha)$ and $(g, h, \bar \beta)$ are monotone-preserving morphisms, 
then $(f, h, \bar \beta \circ \bar \alpha)$ is a monotone-preserving morphism.
\end{prop}
\begin{proof}
Suppose $f : \bar P \to \Z$, $g : \bar Q \to \Z$, 
and $h : \bar R \to \Z$.
For all $I \in \bar R$, $h(I) = g(I^\star)$ where $I^\star := \max \bar \beta^{-1}[\bot, I]$.
Furthermore, $g(I^\star) = f(I^{\star \star})$ where $I^{\star \star} := \max \bar \alpha^{-1}[\bot, I^\star]$.
Since $I^{\star \star} = \max (\bar \beta \circ \bar \alpha)^{-1} [\bot, I]$, we have that
$h(I) = f(I^{\star \star})$.
Thus the composition $(f, h, \bar \beta \circ \bar \alpha)$ is a monotone-preserving morphism.
\end{proof}

\begin{defn}
Let $\Mon$ be the category consisting of monotone integral functions $f : \bar P \to \Z$, over all
finite metric lattices $P$, and monotone-preserving morphisms.
We call $\Mon$ the \define{category of  monotone functions}.
\end{defn}

\subsection{Birth-Death Functor}

Fix a field $\field$.
Let $\Vec$ be the category of finite-dimensional $\field$-vector spaces and 
$\Ch(\Vec)$ the category of chain complexes over $\Vec$.
Let $\Cfunc_\bullet : \Delta K \to \Ch(\Vec)$ be the functor that assigns 
to every subcomplex its simplicial chain complex
and to every inclusion of subcomplexes the induced inclusion of chain complexes.
For every object $\Ffunc : P \to \Delta K$ in $\Fil(K)$, we get a $P$-filtered
chain complex $\Cfunc_\bullet \Ffunc : P \to \Ch(\Vec)$ whose total
chain complex is $\Cfunc_\bullet \Ffunc (\top)$.
For all dimensions $i$, denote by 
$\Zfunc_i \Ffunc : P \to \Vec$
the functor that assigns to every $a \in P$ the subspace of $i$-cycles in
$\Cfunc_\bullet \Ffunc(a)$ and assigns to every $a \leq b$ the canonical inclusion
of $\Zfunc_i \Ffunc(a) \mono \Zfunc_i \Ffunc(b)$.
For all dimensions $i$, denote by $\Bfunc_i \Ffunc : P \to \Vec$
the functor that assigns to every $a \in P$ the subspace of $i$-boundaries in
$\Cfunc_\bullet \Ffunc(a)$ and to all $a \leq b$ the canonical inclusion
of $\Bfunc_i \Ffunc(a) \mono \Bfunc_i \Ffunc(b)$.
In summary, for all $a \leq b$ in $P$, we have following commutative diagram
of inclusions between cycles and boundaries:
	\begin{equation*}
	\begin{tikzcd}
	\Bfunc_i \Ffunc(a) \ar[r, hookrightarrow] \ar[d, hookrightarrow]
	& \Bfunc_i \Ffunc(b) \ar[r, hookrightarrow]  
	\ar[d, hookrightarrow] & \Bfunc_i \Ffunc(\top) \ar[d, hookrightarrow] \\
	\Zfunc_i \Ffunc(a) \ar[r, hookrightarrow] & \Zfunc_i \Ffunc(b) \ar[r, hookrightarrow]
	& \Zfunc_i \Ffunc(\top).
	\end{tikzcd}
	\end{equation*}

\begin{defn}
Let $\Ffunc : P \to \Delta K$ be an object of $\Fil(K)$.
For every interval $[a, b] \in \bar P$, where $b \neq \top$,  let
$$\BD_i \Ffunc [a,b] := \dim \big( \Zfunc_i \Ffunc(a) \cap \Bfunc_i \Ffunc(b) \big)$$
where the intersection is taken inside $\Zfunc_i \Ffunc(\top)$.
For all other intervals $[a, \top]$, let 
$$\BD_i \Ffunc [a,\top] := \dim \Zfunc_i \Ffunc(a).$$
The \define{$i$-th birth-death function} of $\Ffunc$ is the function $f_i : \bar P \to \Z$
that assigns to every interval $[a,b]$ the integer $\BD_i \Ffunc [a,b]$.
\end{defn}

The reason we force $\BD_i \Ffunc [a,\top]$ to $\dim \Zfunc_i \Ffunc(a)$ 
instead of $\dim \Zfunc_i \Ffunc(a) \cap \Bfunc_i \Ffunc (\top)$ is because we want all cycles
to be dead by $\top$.
Otherwise, the persistence diagram for $\Ffunc$ (see Definition~\ref{defn:diagram})  would not see cycles
that are born and never die.

\begin{prop}
Let $\Ffunc : P \to \Delta K$ be an object in $\Fil(K)$ and
$f_i : \bar P \to \Z$ its $i$-th birth-death function.
Then $f_i$ is monotone.
\label{prop:filtomono}
\end{prop}
\begin{proof}
For any two intervals $I \preceq J$ in $\bar P$, we must show that
$f_i(I) \leq f_i(J)$.
Suppose $I = [a,b]$ and $J = [c,d]$ and $d \neq \top$.
Then $\Zfunc_i \Ffunc(a) \subseteq \Zfunc_i \Ffunc(c)$ and 
$\Bfunc_i \Ffunc(b) \subseteq \Bfunc_i \Ffunc(d)$.
Thus $\Zfunc_i \Ffunc(a) \cap \Bfunc_i \Ffunc(b)$
is a subspace of $\Zfunc_i \Ffunc(c) \cap \Bfunc_i \Ffunc(d)$,
and therefore $\BD_i \Ffunc [a,b] \leq \BD_i \Ffunc [c,d]$.
For $J = [c, \top]$, $\BD_i \Ffunc [a,b] \subseteq \Zfunc_i \Ffunc(c)$,
and therefore $\BD_i \Ffunc [a,b] \leq \BD_i \Ffunc [c,\top]$.
\end{proof}

\begin{prop}
Let $(\Ffunc, \Gfunc, \alpha)$ be a morphism in $\Fil(K)$ and
$f_i$ and $g_i$ the $i$-th birth-death functions of $\Ffunc$ and $\Gfunc$, respecively.
Then $(f_i, g_i, \bar \alpha)$ is a morphism in~$\Mon$.
\label{prop:filtomono2}
\end{prop}
\begin{proof}
Suppose $\Ffunc : P \to \Delta K$ and $\Gfunc : Q \to \Delta K$.
By definition of morphism in $\Fil(K)$, $\Gfunc(a) = \Ffunc(a^\star)$, for all $a \in Q$,
where $a^\star = \max \bar \alpha^{-1}[\bot, a]$.
For all intervals $I \in \bar Q$, let $I^\star := \max \bar \alpha^{-1} [\bot, I]$.
If $I = [a,b]$, then $I^\star = [a^\star, b^\star]$ where $b^\star = \max \bar \alpha^{-1} [\bot, b]$.
The definition of a filtration-preserving morphism implies the following canonical isomorphisms of chain complexes:
	\begin{align*}
	\Cfunc_\bullet \Gfunc(b) \cong \Cfunc_\bullet \Ffunc(b^\star) && 
	\Cfunc_\bullet \Gfunc(\top) \cong \Cfunc_\bullet \Ffunc(\top) &&
	\Cfunc_\bullet \Gfunc(a) \cong \Cfunc_\bullet \Ffunc(a^\star),
	\end{align*}
which, in turn, implies canonical isomorphisms $\Zfunc_\bullet \Gfunc(a) \cong \Zfunc_\bullet \Ffunc(a^\star)$
and $\Bfunc_\bullet \Gfunc(b) \cong \Bfunc_\bullet \Ffunc(b^\star)$.
We have $\BD_i \Gfunc (I) = \BD_i \Ffunc(I^\star)$
and therefore $g_i(I) = f_i(I^\star)$.
\end{proof}

By Propositions \ref{prop:monotone_composition}, \ref{prop:filtomono} and \ref{prop:filtomono2}, the assignment
to each object in $\Fil(K)$ its birth-death monotone function is functorial.

\begin{defn}
Let $\BD_i : \Fil(K) \to \Mon$ be the functor that assigns to every filtration 
its $i$-th birth-death
monotone function and to every filtration-preserving morphism
the induced monotone-preserving morphism.
We call $\BD_i$ the \define{$i$-th birth-death functor}.
\end{defn}

\begin{ex}
The functor $\BD_1$ applied to the filtration-preserving morphism $(\Ffunc, \Gfunc, \alpha)$
in Example \ref{ex:three} is the monotone-preserving morphism $(f,g, \bar \alpha)$ in Example \ref{ex:four}.
\end{ex}

%
%

\section{Integral Functions}
\label{sec:four}

We now define the category of integral functions over finite metric lattices $\Fnc$ and construct
the M\"obius inversion functor $\M : \Mon \to \Fnc$.

\begin{defn}
\label{defn:charge_preserving}
Let $P$ and $Q$ be finite metric lattices and let $\sigma : \bar P \to \Z$ and $\tau : \bar Q \to \Z$ be two
integral functions on their lattice of intervals.
Note that $\sigma$ and $\tau$ are not required to be monotone.
A \define{charge-preserving morphism} is a triple $(\sigma, \tau, \bar \alpha)$ where
$\sigma : \bar P \to \Z$ and $\tau : \bar Q \to \Z$ are integral functions and 
$\bar \alpha: \bar P \to \bar Q$ is a bounded lattice function induced by a bounded lattice function
$\alpha : P \to Q$ satisfying the following axiom.
For all $I \in \bar Q$ with $I \neq [q,q]$, 
	\begin{equation}
	\label{eq:charge_preserving}
	\tau(I) = \sum_{J \in \bar \alpha^{-1}(I)} \sigma(J).
	\end{equation}
If $\bar \alpha^{-1}(I)$ is empty, then we interpret the sum as $0$.
\end{defn}

\begin{rmk}
\label{rmk:signed_measure}
Our definition of a charge-preserving morphism is related to the definition of a 
morphism between signed measures.
Let $(X, \Sigma_X)$ and $(Y, \Sigma_Y)$ be measurable spaces,
$\phi : (X, \Sigma_X) \to (Y, \Sigma_Y)$ a measurable map, and 
$\mu : \Sigma_X \to \Rspace$ a signed measure.
Then the pushforward of $\mu$ along $f$ is the signed measure
$\phi_{\#} \mu : \Sigma_Y \to \Rspace$ defined as $\phi_{\#} \mu(U) := \mu \big( \phi^{-1}(U) \big)$.
In the category of signed measures, a morphism from $(X, \Sigma_X, \mu)$ to $(Y, \Sigma_Y, \nu)$
is a measurable map $\mu : \Sigma_X \to \Sigma_Y$ such that $\phi_{\#} \mu = \nu$.
\end{rmk}

\begin{ex}
\label{ex:five}
See Figure \ref{fig:fnc_map} for integral functions $\sigma$ and $\tau$ on the lattices of
intervals $\bar P$ and $\bar Q$, respectively, from Example \ref{ex:three}.
The triple $(\sigma, \tau, \bar \alpha)$, where $\bar \alpha : \bar P \to \bar Q$ is from the same example,
is a charge-preserving morphism.
	\begin{figure}
	\centering
	\includegraphics{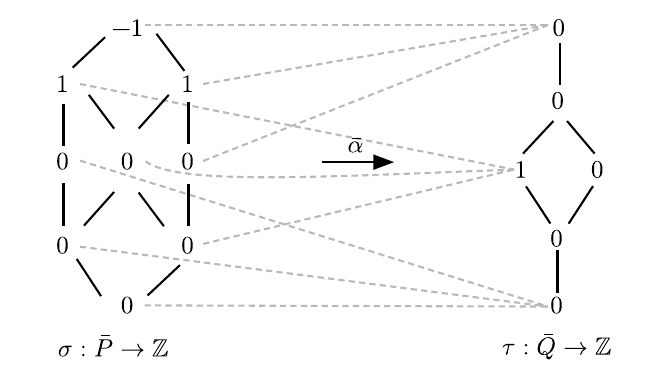}
	\caption{Two integral functions $\sigma$ and $\tau$ on $\bar P$ and $\bar Q$,
	respectively, from Example~\ref{ex:three}. The bounded lattice function
	$\bar \alpha : \bar P \to \bar Q$ from the same figure is a charge-preserving
	morphism from $\sigma$ to $\tau$.}
	\label{fig:fnc_map}
	\end{figure}
\end{ex}

\begin{prop}
\label{prop:charge_composition}
If $(\sigma, \tau, \bar \alpha)$ and $(\tau, \upsilon, \bar \beta)$ are charge-preserving morphisms, 
then $(\sigma, \upsilon, \bar \beta \circ  \bar \alpha)$ is a charge-preserving morphism.
\end{prop}
\begin{proof}
Suppose $\sigma : \bar P \to \Z$, $\tau : \bar Q \to \Z$, 
and $\upsilon : \bar R \to \Z$.
For all $I \in \bar R$ that is not of the form $[r,r]$,
	$$\upsilon(I) = \sum_{J \in  \bar \beta^{-1} (I)} \tau(J) = \sum_{J \in  \bar \beta^{-1}(I)} \sum_{K \in  \bar \alpha^{-1}(J)}
	\sigma(K) = \sum_{K \in ( \bar \beta \circ  \bar \alpha)^{-1}(I)} \sigma(K).$$
Note that since $\bar \beta$ is induced by a bounded lattice function $\beta : Q \to R$,
$J \in \bar \beta^{-1}(I)$ cannot be of the form $[q,q]$.
\end{proof}

\begin{defn}
Let $\Fnc$ be the category whose objects are integral functions $\sigma : \bar P \to \Z$, over all
finite metric lattices $P$, and whose morphisms are charge-preserving morphisms.
We call $\Fnc$ the \define{category of integral functions}.
\end{defn}

\subsection{M\"obius Inversion Functor}

Given any monotone integral function $f : \bar P \to \Z$ of $\Mon$,
there is a unique integral function $\sigma : \bar P \to \Z$ such that
	\begin{equation}
	f(J) = \sum_{I \in \bar P: I \preceq J} \sigma (I)
	\label{eq:mobius}
	\end{equation}
for all $J \in \bar P$ \cite{Rota_Russian, rota64}.
The function $\sigma$ is the called the \emph{M\"obius inversion} of $f$.

\begin{prop}
\label{prop:mobius_inver}
Let $(f, g, \bar \alpha)$ be a morphism in $\Mon$, and let $\sigma$ and $\tau$
be the M\"obius inversions of $f$ and $g$, respectively. 
Then $(\sigma, \tau, \bar \alpha)$
is a morphism in $\Fnc$.
\end{prop}
\begin{proof}
Suppose $f : \bar P \to \Z$ and $g : \bar Q \to \Z$.
We show that 
$$\tau(J) = \sum_{K \in \bar \alpha^{-1}(J)} \sigma(K)$$
for \emph{all} $J \in \bar Q$, and thus $(\sigma, \tau, \bar \alpha)$ is a charge-preserving morphism.
The proof is by induction on the finite metric lattice $\bar Q$.
By Proposition \ref{prop:maximal}, the pre-image $ \bar \alpha^{-1} [\bot, J]$ has a unique
maximal element $J^\star$, and $f(J^\star) = g(J)$ by definition of a morphism in $\Mon$.

For the base case, suppose $J = \bot$.
Then by Equation (\ref{eq:mobius}), $g(J) = \tau(J)$.
By definition of a morphism in $\Mon$, $g(J) = f(J^\star)$.
By Equation (\ref{eq:mobius}), 
$$f(J^\star) = \sum_{K \leq J^*} \sigma(K) = \sum_{K\in \bar \alpha^{-1}(J)} \sigma(K)$$
thus proving the base case.

For the inductive step, suppose $\tau(I) = \sum_{K \in \bar \alpha^{-1}(I)} \sigma(K)$,
for all $I \prec J$.
Then
	\begin{align*}
	\tau(J) &= \sum_{I\in \bar Q: I \preceq J} \tau (I) - \sum_{I \in \bar Q: I \prec J} \tau(I)  \\
	&= g(J) - \sum_{I \in \bar Q: I \prec J} \tau(I) & \text{by Equation (\ref{eq:mobius})}\\
	&= g(J) - \sum_{I \in \bar Q: I \prec J} \sum_{K \in \bar \alpha^{-1}(I) } \sigma(K) & \text{by Inductive Hypothesis} \\
	&= f(J^\star) - \sum_{K \in \bar P: \bar \alpha(K) \prec J} \sigma(K) \\
	&= \sum_{K \in \bar P : K \preceq J^\star} \sigma(K) - \sum_{K \in \bar P: \bar \alpha(K) \prec J} \sigma(K) & \text{by Equation (\ref{eq:mobius})}\\
	&= \sum_{K \in \bar P : \bar \alpha(K) \preceq J} \sigma(K) -
	\sum_{K \in \bar P : \bar \alpha(K) \prec J} \sigma(K)  \\
	&= \sum_{K \in \bar P : \bar \alpha(K) = J} \sigma(K)
	= \sum_{K \in \bar \alpha^{-1}(J)} \sigma(K).
 	\end{align*}
\end{proof}

By Propositions \ref{prop:charge_composition} and \ref{prop:mobius_inver},
the assignment to every object in $\Mon$ its M\"obius inversion is functorial. 

\begin{defn}
Let $\M : \Mon \to \Fnc$ be the functor that assigns to every monotone function its M\"obius inversion
and to every monotone-preserving morphism the induced charge-preserving morphism.
We call $\M$ the \define{M\"obius inversion functor}.
\end{defn}

\begin{ex}
The functor $\M$ applied to the monotone-preserving morphism $(f, g, \bar \alpha)$ in Example
\ref{ex:four} is the charge-preserving morphism $(\sigma, \tau, \bar \alpha)$
in Example \ref{ex:five}.
\end{ex}

\section{Edit Distance}
\label{sec:five}
We now define the edit distance in each of the three categories $\Fil(K)$, $\Mon$, and $\Fnc$
and show that the two functors $\BD_\ast$ and $\M$ are $1$-Lipschitz.
Denote by $\star$ the metric lattice consisting of just one element.

\subsection{Distance Between Filtrations}

A \emph{path} between two filtrations $\Ffunc$ and $\Hfunc$ in $\Fil(K)$ is a finite sequence
	\begin{equation*}
	\begin{tikzcd}
	\Ffunc \ar[r, leftrightarrow, "\alpha_1"] & \Gfunc_1 \ar[r, leftrightarrow, "\alpha_2"] & \cdots 
	\ar[r, leftrightarrow, "\alpha_{n-1}"] & \Gfunc_{n-1} \ar[r, leftrightarrow, "\alpha_n"] & \Hfunc
	\end{tikzcd}
	\end{equation*}
where $\leftrightarrow$ denotes a filtration-preserving morphism in either direction.
The \emph{length} of a path is the sum $\sum_{i=1}^n || \alpha_i ||$
of the distortions of all the bounded lattice functions.
Again, $|| \alpha_i ||$ might be infinite and so the length of a path might be infinite.
Note that the filtration $\Omega : \star \to \Delta K$ is terminal in $\Fil(K)$.
This implies that any two filtrations in $\Fil(K)$ are connected by a path.

\begin{defn}
\label{defn:edit_fil}
The \define{edit distance} $\Dist_{\Fil(K)}(\Ffunc,\Hfunc)$ between any two filtrations
in $\Fil(K)$ 
is the length of the shortest path between $\Ffunc$ and $\Hfunc$.
\end{defn}

\subsection{Distance Between Monotone Integral Functions}

A \emph{path} between two monotone functions $f$ and $h$ in $\Mon$ is a finite sequence
	\begin{equation*}
	\begin{tikzcd}
	f \ar[r, leftrightarrow, " \bar \alpha_1"] & g_1 \ar[r, leftrightarrow, " \bar \alpha_2"] & \cdots 
	\ar[r, leftrightarrow, " \bar \alpha_{n-1}"] & g_{n-1} \ar[r, leftrightarrow, " \bar \alpha_n"] & h
	\end{tikzcd}
	\end{equation*}
where $\leftrightarrow$ denotes a monotone-preserving morphism in either direction.
The \emph{length} of a path is the sum $\sum_{i=1}^n ||  \bar \alpha_i ||$
of the distortions of all the bounded lattice functions.
Suppose $f [\top, \top] = n$, and let $e : \bar \star \to \Z$ be the monotone integral function
where $e [\star, \star] = n$.
Then there is a unique monotone-preserving morphism from $f$ to $e$.
Thus there is a path
between any two monotone-integral functions $f$ and $h$ such that $f [\top, \top] = h[\top, \top]$.

\begin{defn}
\label{defn:edit_mon}
The \define{edit distance} $\Dist_\Mon(f,h)$ between any two monotone functions in $\Mon$
is the length of the shortest path between $f$ and $h$.
If there are no paths, then we let $\Dist_\Mon (f,h) = \infty$.
\end{defn}

\begin{lem}
\label{lem:first}
Let $\Ffunc$ and $\Gfunc$ be two objects of $\Fil(K)$.
Then for every dimension $i$, 
$$\Dist_\Mon ( \BD_i \Ffunc, \BD_i \Gfunc) \leq \Dist_{\Fil(K)} (\Ffunc, \Gfunc).$$
\end{lem}
\begin{proof}
Suppose $\Dist_{\Fil(K)} (\Ffunc, \Gfunc) = \ee$.
Then there is a path in $\Fil(K)$ between $\Ffunc$ and $\Gfunc$ with length~$\ee$.
Apply the functor $\BD_i$ to this path and the result is a path in $\Mon$
between $\BD_i \Ffunc$ and $\BD_i \Gfunc$ and its length, by Proposition \ref{prop:distortion}, is also $\ee$.
Since the distance between the two monotone functions is defined as the length of the
shortest path between them, we have the desired inequality.
\end{proof}

\subsection{Distance Between Integral Functions}

A \emph{path} between two integral functions $\sigma$ and $\tau$ in $\Fnc$ is a finite sequence
	\begin{equation*}
	\begin{tikzcd}
	 \sigma \ar[r, leftrightarrow, " \bar \alpha_1"] &  \theta_1 \ar[r, leftrightarrow, " \bar \alpha_2"] & \cdots 
	\ar[r, leftrightarrow, " \bar \alpha_{n-1}"] &  \theta_{n-1} \ar[r, leftrightarrow, " \bar \alpha_n"] &  \tau
	\end{tikzcd}
	\end{equation*}
where $\leftrightarrow$ denotes a charge-preserving morphism in either direction.
The \emph{length} of path is the sum $\sum_{i=1}^n ||  \bar \alpha_i ||$
of the distortions of all the bounded lattice functions.
Note that any integral function $\omega : \bar \star \to \Z$ is terminal in $\Fnc$; see Definition \ref{defn:charge_preserving}.
This means that any two integral functions in $\Fnc$ are connected by a path, but
this path may have infinite length.

\begin{defn}
\label{defn:edit_fnc}
Define the distance $\Dist_\Fnc( \sigma, \tau)$ between any two integral functions 
in $\Fnc$
as the length of the shortest path between $ \sigma$ and $ \tau$.
\end{defn}

\begin{lem}
\label{lem:second}
Let $f$ and $g$ be two objects of $\Mon$.
Then $\Dist_\Fnc \big( \M (f), \M (g) \big) \leq \Dist_\Mon (f, g)$. 
\end{lem}
\begin{proof}
Suppose $\Dist_\Mon (f, g) = \ee$.
Then there is a path in $\Mon$ between $f$ and $g$ with length~$\ee$.
Apply the functor $\M$ to this path and the result is a path in $\Fnc$
between $\M(f)$ and $\M(g)$ and its length
is also $\ee$.
Since the distance between the two functions is defined as the length of the
shortest path between them, we have the desired inequality.
\end{proof}

\section{Persistence Diagrams}
\label{sec:persistence_diagrams}
The pieces established in the last four sections fit together into the following pipeline
of $1$-Lipschitz functors:
	\begin{equation*}
	\begin{tikzcd}
	\Fil(K) \ar[rr, "\BD_\ast"] && \Mon \ar[rr, "\M"] && \Fnc.
	\end{tikzcd}
	\end{equation*}The birth-death functor $\BD_\ast$ assigns to an object
$\Ffunc : P \to \Delta K$ of $\Fil(K)$ a monotone integral function
$f_i := \BD_i ( \Ffunc) : \bar P \to \Z$ for every dimension $i$.
The value of $f_i$ on an interval $[a,b] \subseteq P$ is
the dimension of the $\field$-vector space of $i$-cycles that appear by $a$ and become boundaries by~$b$.
The M\"obius inversion functor $\M$ assigns to $f_i$ its M\"obius inversion, which is an 
integral function $\sigma:= \M (f_i) : \bar P \to \Z$.

	\begin{defn}
	\label{defn:diagram}
	Let $P$ be a finite metric lattice and $\Ffunc : P \to \Delta K$ a filtered simplicial complex indexed
	over $P$.
	The \define{$i$-th persistence diagram} of $\Ffunc$ is the integral
	function $\sigma := \M \circ \BD_i (\Ffunc) : \bar P \to \Delta K$.
	\end{defn}
	
\begin{ex}
Consider the filtrations $\Ffunc$ and $\Gfunc$
in Example \ref{ex:two}.
Their $1$-dimensional persistence diagrams are the integral functions 
$\sigma$ and $\tau$, respectively, in Example \ref{ex:five}. 
The integer $\sigma [b,d] = 1$ represents the $1$-cycle that is born at $b$,
and the integer $\sigma[c,d] = 1$ represents the $1$-cycle that is born at $c$.
The integer $\sigma [d,d] = -1$, represents the $1$-cycle that was born twice but contributes
to just one dimension of the total cycle space.
\end{ex}

\begin{ex}
Consider the example of a filtration $\Ffunc : P \to \Delta K$ in Figure \ref{fig:zero_d}
where $P$ is the lattice from Example \ref{ex:one} and $K$ is the $1$-simplex.
Recall $\bar P$ in Example \ref{ex:three}.
Drawn are its zeroth birth-death function $f := \BD_0 \circ \Ffunc : \bar P \to \Z$ and its zeroth persistence diagram 
$\sigma := \M \circ \BD_0 : \bar P \to \Z$.
The integer $\sigma [a,d] = 1$ represents the $0$-cycle that is born at $a$,
and the integer $\sigma [b,d] = 1$ represents the $0$-cycle that is born at $b$.
The integer $\sigma [c,c] = 1$ represents the $0$-cycle that is born at $c$ and dies immediately.
The integer $\sigma [d,d] = -1$ represents the $0$-cycle that was born twice but contributes to just
one dimension of the total cycle space.
	\begin{figure}
	\centering
	\includegraphics{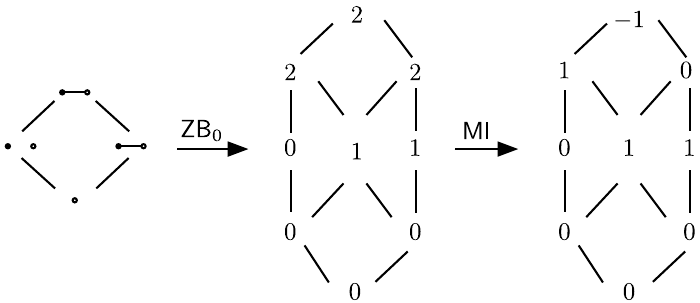}
	\caption{Filtration of the $1$-simplex, its zeroth birth-death function,
	 and its zeroth persistence diagram.
	 Note that since $d$ is the top element of $P$, every interval of the form
	 $[x, d]$ is assigned the dimension of the $0$-cycle space at $x$.}
	\label{fig:zero_d}
	\end{figure}
\end{ex}

Our main theorem follows immediately from Lemmas \ref{lem:first} and \ref{lem:second}.

\begin{thm}[Stability]
\label{thm:main}
Let $\Ffunc : P \to \Delta K$ and $\Gfunc : Q \to \Delta K$ be
two filtrations of a finite simplicial complex $K$ indexed by finite metric lattices, 
and $\sigma : \bar P \to \Z$ and $\tau : \bar Q \to \Z$ their $i$-th persistence diagrams.
Then $\Dist_\Fnc (\sigma, \tau) \leq \Dist_{\Fil(K)}(\Ffunc, \Gfunc)$.
\end{thm}

\section{Classical Persistent Homology}
\label{sec:ordinary}
We now relate our definitions to that of classical persistent homology.
First, we show that our definition of the persistence diagram is the same as the original
definitions of \cite{CSEdH} and \cite{size_theory,10.1117/12.279674}.
Second, we show that the bottleneck distance between classical persistence
diagrams is strongly equivalent to the edit distance.

\subsection{Classical Persistence Diagrams}
\label{subsec:classical_diagrams}

Fix a finite $1$-parameter filtration $K_{r_1} \subseteq K_{r_1} \subseteq \cdots \subseteq K_{r_n} = K$
of a finite simplicial complex~$K$ indexed by real the numbers $r_1 < \cdots < r_n$.
Let $P := 1 <  \cdots < n < \infty$ be the totally ordered lattice with $n+1$ elements
with $d_P(a,b) = | r_a - r_b |$, for $b \neq \infty$,
and $d_P(a,\infty) = \infty$.
Let $\Ffunc : P \to \Delta K$ be the filtration that assigns to every $a \in P \setminus \{ \infty \}$ 
the subcomplex $K_{r_a}$ and to $\infty$ the total complex $K$.
Cohen-Steiner, Edelsbrunner, and Harer define the $i$-th persistence diagram of this filtration as the integral
function $\sigma_i :  \bar P \to \Z$ defined as follows.
For $a < b \neq \infty$,  $\sigma_i [a,b]$ is the following signed sum
of ranks:
	\begin{align*}
	\sigma_i [a,b] :=& \; \rank \Hfunc_i F ( a \leq b-1 )  - \rank \Hfunc_i F ( a-1 \leq b-1 ) 
	- \rank \Hfunc_i F ( a \leq b  ) + \rank \Hfunc_i F ( a-1 \leq b) \\
	=& \dim \dfrac{\Zfunc_i F (a)}{\Zfunc_i F(a) \cap \Bfunc_i F(b-1)} - 
	\dim \dfrac{\Zfunc_i F (a-1)}{\Zfunc_i F(a-1) \cap \Bfunc_i F(b-1)} \\
	& -\dim \dfrac{\Zfunc_i F (a)}{\Zfunc_i F(a) \cap \Bfunc_i F(b)} +
	\dim \dfrac{\Zfunc_i F (a-1)}{\Zfunc_i F(a-1) \cap \Bfunc_i F(b)} \\
	=& \dim \Zfunc_i F (a) - \dim \big( \Zfunc_i F(a) \cap \Bfunc_i F(b-1) \big) - 
	\dim \Zfunc_i F (a-1) + \dim \big( \Zfunc_i F(a-1) \cap \Bfunc_i F(b-1) \big) \\
	& -\dim \Zfunc_i F (a) + \dim \big( \Zfunc_i F(a) \cap \Bfunc_i F(b) \big) +
	\dim \Zfunc_i F (a-1) \\
	& - \dim \big( \Zfunc_i F(a-1) \cap \Bfunc_i F(b) \big) \\
	=& - \dim \big( \Zfunc_i F(a) \cap \Bfunc_i F(b-1) \big) + 
	\dim \big( \Zfunc_i F(a-1) \cap \Bfunc_i F(b-1) \big) \\
	& + \dim \big( \Zfunc_i F(a) \cap \Bfunc_i F(b) \big)
	- \dim \big( \Zfunc_i F(a-1) \cap \Bfunc_i F(b) \big) \\
	=& - \BD_i \Ffunc [a,b-1] + \BD_i \Ffunc [a-1,b-1] + \BD_i \Ffunc [a,b] - 
	\BD_i \Ffunc [a-1,b].
	\end{align*}
For $a < b = \infty$, $\sigma[a,\infty]$ is the following signed sum of ranks:
\begin{align*}
	\sigma_i [a, \infty] :=& \; \rank \Hfunc_i F ( a \leq \infty )  - \rank \Hfunc_i F ( a-1 \leq \infty ) \\
	=& \dim \dfrac{\Zfunc_i F (a)}{\Zfunc_i F(a) \cap \Bfunc_i F(n)} - 
	\dim \dfrac{\Zfunc_i F (a-1)}{\Zfunc_i F(a-1) \cap \Bfunc_i F(n)} \\
	=& \dim \Zfunc_i F (a) - \dim \big( \Zfunc_i F(a) \cap \Bfunc_i F(n) \big) - 
	\dim \Zfunc_i F (a-1) + \dim \big( \Zfunc_i F(a-1) \cap \Bfunc_i F(n) \big) \\
	=& \BD_i \Ffunc [a,\infty] - \BD_i \Ffunc [a,n] - \BD_i \Ffunc [a-1,\infty] + 
	\BD_i \Ffunc [a-1,n].
	\end{align*}
However, in this paper we define the persistence diagram of $F$ as 
$\tau := \M \circ \BD_\ast (F) : \bar P \to \Z$;
see Definition \ref{defn:diagram}.
It turns out that the two are the same.
Since $P$ is totally ordered, the M\"obius inversion of $\BD_\ast \Ffunc$ has the following simple
formula for any $a \leq b$:
	$$\tau[a,b] := \BD_\ast F [a,b] - \BD_\ast F [a-1,b] - \BD_\ast F [a,b-1] + \BD_\ast F [a-1,b-1].$$
We see that for $a < b$, $\sigma [a,b] = \tau[a,b]$.

\subsection{Bottleneck Distance}
\label{subsec:bottleneck}

We prove that the bottleneck distance defined between classical persistence diagrams is strongly equivalent
to the edit distance. 
Let $P$ and $Q$ be finite, totally ordered metric lattices. 
In order to define the bottleneck distance between two integral functions 
$\sigma: \bar P \to \Z$ and $\tau: \bar Q \to \Z$, we need isometric, monotone embeddings of $P$ and $Q$ into 
the totally ordered lattice $\Rspace$. 
This is problematic since the edit distance does not depend on the embedding while the bottleneck distance does. 
We fix this issue by requiring that $\bot_P$ and $\bot_Q$ map to $0\in \Rspace$ under the embeddings. 
We identify elements of $P$ and $Q$ with their images in $\Rspace$ under the assumed embeddings. 
This section culminates in a proof of the following theorem.

\begin{thm}\label{thm:bottleneckequivalence}
Let $P$ and $Q$ be finite, totally ordered metric lattices with an isometric, 
monotone embedding into $\Rspace$ such that $\bot_P = \bot_Q = 0$.
Let $\sigma : \bar P \to \Z$ and $\tau : \bar Q \to \Z$ be two non-negative integral functions.
Then $d_B(\sigma, \tau) \leq d_{\Fnc} (\sigma, \tau) \leq 2 d_B(\sigma, \tau)$.
\end{thm}

\begin{defn}
For any two intervals $[a,b], [c,d] \subseteq \Rspace$, let 
$$\big | \big | [a,b] - [c,d] \big | \big |_\infty := \max \big\{ |a-c| , |b-d| \big \}.$$ 
Addition and scalar multiplication of intervals is defined componentwise by 
$[a,b]+[c,d]=[a+c,b+d]$ and $x[a,b]=[xa,xb]$ for any $x\in \Rspace^{\geq 0}$.
\end{defn}
	
\begin{defn}
A \define{matching} between two non-negative integral functions $\sigma: \bar P \to \Z$ and $\tau : \bar Q \to \Z$ is a non-negative
map $\gamma: \bar P \times \bar Q \to \Z$ satisfying 
	\begin{align*}
	\sigma(I) & = \sum_{J \in \bar Q} \gamma(I,J) \textrm{ for all }I\neq [p,p] \in \bar P \\
	\tau(J) & = \sum_{I \in \bar P} \gamma(I,J) \textrm{ for all } J \neq [q,q] \in \bar Q.
	\end{align*}
The \define{norm} of a matching $\gamma$ is
$$ ||\gamma|| := \max_{\big \{I \in \bar P , J \in \bar Q \,\big |\, \gamma(I,J)> 0 \big \}} 
 ||I-J||_\infty. $$
 A matching $\gamma$ is an \define{$\ee$-matching} if $|| \gamma || = \ee$.
The \define{bottleneck distance} between $\sigma$ and $\tau$ is
$$d_B(\sigma, \tau) := \min_\gamma ||\gamma||$$
over all matchings $\gamma$ between $\sigma$ and $\tau$.
\end{defn}
	
\begin{prop}\label{prop:Interpolation}
Let $\sigma : \bar P \to \Z$ and $\tau: \bar Q \to \Z$ be non-negative integral functions and~$\gamma$ a matching between $\sigma$ and $\tau$. Then $\gamma$ induces a 1-parameter family of integral functions $\{ \upsilon_t \}_{t\in [0,1]}$ with $\upsilon_0 = \sigma$ and $\upsilon_1 = \tau$.
\end{prop}
\begin{proof}
Let $\bar S_t:= \big \{ (1-t)I +tJ \,\big|\, I\in \bar P,\, J\in \bar Q, \text{ and } \gamma(I,J)>0 \big \}$.
Define $\upsilon_t: \bar S_t \to \Z$ to be
$$
\upsilon_t(K) := \sum_{\substack{I\in \bar P,J\in \bar Q \\ (1-t)I + tJ = K}} \gamma (I,J).
$$
At $t=0$ this reduces to 
$$
\upsilon_0(K) = \sum_{\substack{I \in \bar P , J \in \bar Q \\ I=K}} \gamma(I,J) = \sum_{J \in \bar Q} \gamma(K,J) = \sigma(K),
$$
for all $K \in \bar{S}_0$,
and similarly $\upsilon_1(I) = \tau(I)$.
\end{proof}

As $t$ varies from 0 to 1, there are only finitely many places where the combinatorial structure of $\upsilon_t$ changes. We call these places critical points; see the following definition. These combinatorial changes occur where endpoints of intervals in $\bar S_t$ cross or, equivalently, where the cardinality of the set of endpoints changes.

\begin{defn}\label{def:Critical}
Let $S_t = \{ w \in \Rspace \,|\, \ [w,x] \text{ or } [x,w] \in \bar S_t \}$ be the set of endpoints of intervals in $\bar S_t$.
A point $t\in [0,1]$ is \define{critical} if for all sufficiently small
$\delta >0$, there exists $s \in (t-\delta , t+\delta )$ with $|S_t| \neq |S_s|$.
\end{defn}

\begin{lem}\label{lem:Critical}
If $t \in [0,1]$ is not a critical point, then for any $K \in \bar S_t$ there is a unique pair of intervals $I \in \bar P$ and $J \in \bar Q$ with $\gamma(I,J) > 0$ and $(1-t)I +tJ = K$.
\end{lem}
\begin{proof}
Suppose $t\in [0,1]$ is not critical and there exists $I , I' \in \bar P$ and $J,J' \in \bar Q$ with $\gamma(I,J)>0$, $\gamma(I',J')>0$ and $(1-t)I +tJ = (1-t) I' +t J'$. Then for any $t'$ sufficiently close to $t$, $(1-t')I +t'J = (1-t') I' +t' J'$. Since the interpolation is linear and two lines that intersect in more than one point must be the same line, it follows that $I=I'$ and $J=J'$.
\end{proof}

\begin{lem}\label{lem:supnorm}
If $\alpha : P \to Q$ is a metric lattice map and $\bar \alpha : \bar P \to \bar Q$ is its induced map on intervals then 
$$\max_{I \in \bar P} || I - \bar\alpha (I) ||_\infty \leq ||\bar \alpha|| \leq 2 \max_{I \in \bar P}|| I - \bar \alpha(I) ||_\infty.$$
\end{lem}
\begin{proof}
First note that by Proposition \ref{prop:distortion}, $||\alpha|| = ||\bar\alpha||$ and since $\bar \alpha$ is induced by $\alpha$, $\max_{I \in \bar P} ||I - \bar \alpha (I)||_\infty = \max_{a\in P} |a - \alpha(a)|$ so the inequality reduces to
$$ \max_{a \in P} |a- \alpha(a)| \leq \max_{a,b \in P} \big| |a-b| - |\alpha(a) - \alpha(b)| \big| \leq 2 \max_{a \in P} |a - \alpha(a)|.
$$ 
Note that since $\bot_P = \bot_Q =0$, each element of $P$ and $Q$ are non-negative.
Assume,  without loss of generality, that $a \geq b$. Then the middle quantity above reduces to 
$$\max_{a \geq b \in P} \big|a-b - \big( \alpha(a) -\alpha(b) \big) \big| = \max_{a \geq b \in P} \big| a-\alpha(a) - \big( b-\alpha(b) \big) \big|$$
Letting $b=0$ yields the first inequality and the triangle inequality yields the second.
\end{proof}

\begin{lem}\label{lem:induced_maps}
If $t\in [0,1]$ is not a critical point and $s\in[0,1]$ is any point with no critical points strictly between $t$ and $s$, then there is a charge-preserving morphism $(\upsilon_t , \upsilon_s,\bar\alpha_{t,s})$ with distortion at most $2\ee|s-t|$.
Here $\ee$ is the norm of the matching $\gamma$ between $\sigma$ and $\tau$.
\end{lem}

\begin{proof}
We start by defining a map $\alpha_{t,s} : S_t \to S_s$. For any $b \in S_t$ note that since $t$ is not critical, there are unique intervals $I \in \bar P$ and $J \in \bar Q$ with $\gamma(I,J) > 0$ and either $(1-t) I +t J = [a,b]$ or $[b,c]$. If $(1-t)I +tJ =[a,b]$ then define $\alpha_{t,s} (b) $ to be the right endpoint of the interval $(1-s)I +s J$. 
Similarly, if $b$ is a left endpoint, then we define $\alpha_{t,s} (b)$ to be the left endpoint of $(1-s)I +s J$. 
This map is order preserving since as $t$ varies, endpoints of intervals only cross at critical points 
and there are no critical points strictly between $t$ and $s$.
 
To prove that $\bar\alpha_{t,s}$ is charge-preserving, observe that

\begin{align*}
\sum_{K \in \bar \alpha_{t,s}^{-1}(L)} \upsilon_t ( K ) &= \sum_{ \substack{ I \in \bar P , J \in \bar Q \\ (1-s) I + s J =L }} \upsilon_t ( (1-t) I +t J ) \\
&= \sum_{ \substack{ I \in \bar P , J \in \bar Q \\ (1-s) I + s J =L }}  \bigg( \sum_{\substack{I'\in \bar P, J'\in \bar Q \\ (1-t)I' + tJ' = (1-t) I +t J}} \gamma (I',J') \bigg) \\
&= \sum_{ \substack{ I \in \bar P , J \in \bar Q \\ (1-s) I + s J =L }} \gamma( I , J ) = \upsilon_s(L)
\end{align*}
where the third equality follows from Lemma \ref{lem:Critical} and the assumption that $t$ is not critical.
The distortion of $\bar\alpha_{t,s}$ is
\begin{align*}
||\bar\alpha_{t,s}|| &\leq 2\max_{K \in \bar S_t}  || K - \bar\alpha_{t,s}(K)||_\infty  \\
&= 2 \max_{\substack{ I \in \bar P, J  \in \bar Q \\ \gamma(I,J) >0 }} ||(1-t)I +tJ - (1-s) I -sJ ||_\infty \\
&= 2|s-t| \max_{\substack{I\in \bar P, J \in \bar Q \\ \gamma(I,J)>0}} ||I - J||_\infty \leq 2\ee|s-t| .
\end{align*}
\end{proof}
	
	\begin{lem}\label{lem:bottleneck1}
	For any non-negative integral functions $\sigma: \bar P \to \Z$ and $\tau: \bar Q \to \Z$ over finite
	sublattices $P , Q\subseteq \Rspace$, $d_\Fnc (\sigma , \tau) \leq 2 d_B(\sigma, \tau)$.
	\end{lem}
 
 \begin{proof}
We show that $d_\Fnc(\sigma,\tau) \leq 2 d_B(\sigma,\tau)$ by showing that an $\ee$-matching between $\sigma$ and $\tau$ induces a path between $\sigma$ and $\tau$ of length at most $2\ee$. For any $\ee$-matching $\gamma$ between $\sigma$ and $\tau$, let $\{\upsilon_t \}_{t\in [0,1]}$ be the interpolation induced by $\gamma$ from Proposition \ref{prop:Interpolation}. Let $\{s_0=0<s_1\cdots< s_n=1\} \subseteq [0,1]$ be the set of critical points of the interpolation and choose $\{t_0<\cdots <t_{n-1}\}\subseteq [0,1]$ with $0<t_0 <s_1 < t_1 \cdots < t_{n-1}<1$. Then the charge-preserving morphisms $\alpha_{t_i,s_{i}}$ 
and $\alpha_{t_i,s_{i+1}}$ from Lemma \ref{lem:induced_maps} form a path between $\sigma$ and $\tau$ with length at most $\sum_{i=0}^{n-1} 2 \ee \big( |t_i - s_{i}| + |t_i - s_{i+1}| \big) = 2\ee$.
 \end{proof}
	
	\begin{lem}\label{lem:bottleneck2}
	For any non-negative integral functions $\sigma: \bar P \to \Z$ and $\tau: \bar Q \to \Z$ over finite sublattices $P , Q \subseteq \Rspace$, $d_\Fnc (\sigma , \tau) \geq d_B(\sigma, \tau)$.
	\end{lem}
	
	\begin{proof}
	To show that $d_B (\sigma, \tau) \leq d_\Fnc (\sigma , \tau)$, it is enough to show that a 
	single charge-preserving morphism induces a matching. Let $(\sigma, \tau,\alpha)$ be a charge-preserving morphism with distortion $\ee$. Define a matching $\gamma$ between $\sigma$ and $\tau$ by
$$\gamma(I,J) := 
\begin{cases}
\sigma(I)	& \text{ if } \alpha(I)=J\\
0	& \text{otherwise}
\end{cases}.
$$

Then we have that for any $J \in \bar Q$
$$
\sum_{I\in \bar P} \gamma(I,J) = \sum_{I \in \alpha^{-1}(J)} \sigma(I) = \tau(J)
$$ and for any $I \in \bar P$
$$
\sum_{J\in \bar Q} \gamma(I,J) = \alpha(I).
$$
Therefore $\gamma$ is a matching. The norm of $\gamma$ is
$$
||\gamma||=\max_{I\in \bar P, J \in \bar Q \, : \, \gamma(I,J)>0} ||I-J||_\infty = \max_{I \in \bar P} ||I - \alpha(I)||_\infty \leq ||\alpha||.
$$
	\end{proof}
	
Theorem \ref{thm:bottleneckequivalence} follows immediately from Lemma \ref{lem:bottleneck1} and Lemma \ref{lem:bottleneck2}.
The following two examples show that the bounds in Theorem \ref{thm:bottleneckequivalence} are tight.

	\begin{ex}
	\label{ex:Matching}
	Let $P = \{ 0 < 1 < 2 < 3 \}$ 
	be a totally ordered metric lattice where the distance between two elements is the
	absolute value of their difference.
	Let $\sigma, \upsilon : \bar P \to \Z$ be two integral functions defined as
	\begin{equation*}
	\sigma [a,b] := 
	\begin{cases}
	1 & \text{if $[a,b] = [0,1], [2,3]$} \\
	0 & \text{otherwise}
	\end{cases} 
	\; \; \; \; \; \; 
	\upsilon [a,b] := 
	\begin{cases}
	1 & \text{if $[a,b] = [0,2], [1,3]$} \\
	0 & \text{otherwise}.
	\end{cases}
	\end{equation*} 
	See Figure \ref{fig:matching}.
	The bottleneck distance, $d_B$, between $\sigma$ and $\upsilon$
	is $d_B (\sigma, \upsilon) = 1$.
	We now compute the edit distance, $d_{\Fnc}$, between $\sigma$ and $\upsilon$.
	Consider a third integral function $\tau : \bar Q \to \Z$ where 
	$Q = \{ 0 < 0.5 < 1 < 1.5 < 2 < 2.5 < 3 \}$ is a finite, totally ordered metric lattice
	where the distance between any two elements is the absolute value of the difference and
	\begin{equation*}
	\upsilon [a,b] := 
	\begin{cases}
	1 & \text{if $[a,b] = [0,1.5], [1.5,3]$} \\
	0 & \text{otherwise}.
	\end{cases} 
	\end{equation*} 
	Let $\alpha : P \to Q$ be the bounded lattice function defined as follows
	\begin{align*}
	\alpha(0) := 0 && \alpha(1) := 1.5 && \alpha(2) := 1.5 && \alpha(3) := 3.
	\end{align*} 
	We now have a pair of charge-preserving morphisms $(\sigma, \tau, \bar \alpha)$
	and $(\upsilon, \tau, \bar \alpha)$.
	Thus $d_{\Fnc}(\sigma, \upsilon) \leq 2 || \bar \alpha || = 2 || \alpha || = 2 (0.5) = 1$.
	Further, this is a shortest path between $\sigma$ and $\upsilon$ in $\Fnc$.
	Therefore $d_{\Fnc}( \sigma, \upsilon) = 1$.

	\begin{figure}
	\centering
	\includegraphics{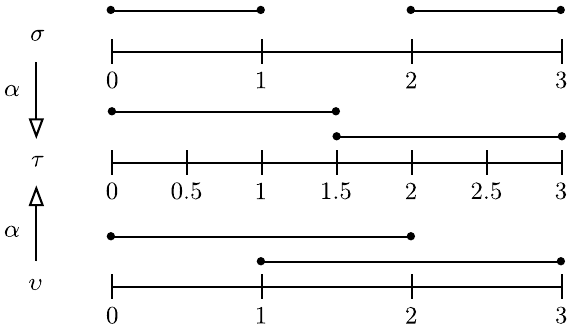}
	\caption{Three integral functions $\sigma, \upsilon: \bar P \to \Z$ and
	$\tau : \bar Q \to \Z$ drawn as barcodes and two charge-preserving morphisms
	$(\sigma, \tau, \bar \alpha)$ and $(\upsilon, \tau, \bar \alpha)$.}
	\label{fig:matching}
	\end{figure}
	\end{ex}

\begin{ex}

Let $P$ be the metric lattice defined in Example \ref{ex:Matching} and $\sigma, \tau : \bar P \to \Z$ be defined as
\begin{equation*}
	\sigma [a,b] := 
	\begin{cases}
	1 & \text{if $[a,b] = [1,2]$} \\
	0 & \text{otherwise}
	\end{cases} 
	\; \; \; \; \; \; 
	\upsilon [a,b] := 
	\begin{cases}
	1 & \text{if $[a,b] = [0,3]$} \\
	0 & \text{otherwise}.
	\end{cases}
	\end{equation*}
See Figure \ref{fig:bot_neq_edit}. The bottleneck distance between $\sigma$ and $\tau$ is 1. We now compute the edit distance $d_\Fnc(\sigma,\tau)$. Let $\alpha:P \to P$ be the bounded lattice function defined by
	\begin{align*}
	\alpha(0) := 0 && \alpha(1) := 0 && \alpha(2) := 3 && \alpha(3) := 3.
	\end{align*} 
The lattice map $\alpha$ induces a charge-preserving morphism $(\sigma,\tau,\bar\alpha)$ with distortion 2. This is the shortest path between $\sigma$ and $\tau$ in $\Fnc$ so $d_\Fnc(\sigma,\tau)=2$.

	\begin{figure}
	\centering
	\includegraphics{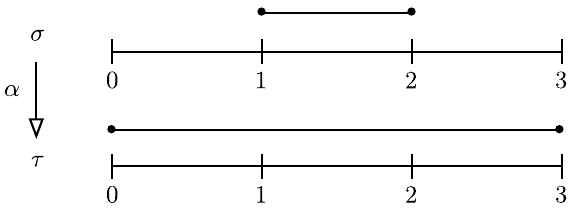}
	\caption{Two integral functions $\sigma, \tau: \bar P \to \Z$ drawn as barcodes and a charge-preserving morphism $(\sigma, \tau, \bar \alpha)$.}
	\label{fig:bot_neq_edit}
	\end{figure}
\end{ex}

\begin{center}
    \Large
    Erratum for ``Edit Distance and Persistence Diagrams Over Lattices''
\end{center}

\begin{abstract}
This erratum corrects a mistake in ``Edit Distance and Persistence Diagrams
Over Lattices'' published in \emph{SIAM J. Algebra Geometry} 6 (2022), pp 134--155. 
This mistake rendered the edit distance between integral functions identically zero.
Here we implement some minor modifications that fix this mistake.
To demonstrate this, we prove a non-trivial lower bound for the edit distance between integral functions.
\end{abstract}

The edit distance between integral functions, as written in the published version of the article, is identically zero; see Example~\ref{ex:zero-distance}.
We thank Luis Scoccola for finding this problem.

\setcounter{section}{10}
\setcounter{defn}{0}
\begin{ex}\label{ex:zero-distance}
    Consider the example in Figure~\ref{fig:ex}.
    Here, we have three totally ordered posets, $P_1$, $P_2$, and $P_3$, each with an integral function $\sigma_i : P_i \to \Z$ described
    by blue and red bars.
    A blue segment indicates an assignment of $+1$ to that interval and a red segment indicates an assignment of $-1$ to that interval.
    The metric on each $P_i$ is the metric inherited from its embedding, as drawn, into the real line.
    The arrows between posets describe bounded lattice functions inducing charge-preserving morphisms between integral functions.
    Thus, we have a path in $\Fnc$ from $\sigma_1$ to $\sigma_3$.
    The distortions of the bounded lattice functions from $P_2$ to $P_1$ and from $P_2$ to $P_3$ are $1$.
    Therefore, the length of this path between $\sigma_1$ and $\sigma_3$ is $2$.
    This process can be refined, replacing $\sigma_2$ with an alternating sequence of intervals that are arbitrarily close, creating paths of arbitrarily small length in $\Fnc$.
    Therefore, $d_\Fnc(\sigma_1, \sigma_3) = 0$.
    \end{ex}

\begin{figure}[h]
    \centering \includegraphics[scale=1.4]{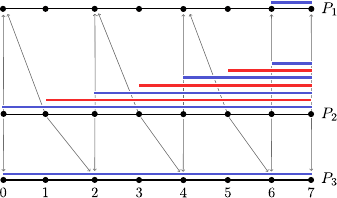}
    \caption{Path in $\Fnc$ whose length can be made arbitrarily small.}
    \label{fig:ex}
\end{figure}

\newpage

To fix this issue, we make the following modifications:
\begin{enumerate}
    \item We assume that for any metric lattice $(P, d_P)$, the distance function $d_P$ is order-preserving when viewed as a function from $P^\op \times P$ to $\mathbb R^{\geq 0}$.
    This assumption is equivalent to the statement that for any $a \leq b \leq c$ in $P$, $d_P(a,b) \leq d_P(a,c)$ and ${d_P(b,c) \leq d_P(a,c)}$.
    This condition is satisfied by most reasonable metric posets of interest in applied topology.
    In particular, every subposet of $\mathbb R^n$ with the inherited metric satisfies this condition.

    \item We modify the category $\Fnc$ to consist of only those integral functions that arise as a M\"obius inversion of an order-preserving function.
    This restriction does not limit our overall pipeline in the slightest as every persistence diagram of a filtration satisfies this condition.

    \item We require morphisms in $\Fnc$ to satisfy the push-forward condition (Equation~\ref{eq: push-forward}) everywhere, including along the diagonal.
    Again, this does not restrict the broader pipeline as every morphism between filtrations induces such a morphism.
    This modification does, however, make $d_\Fnc$ more rigid as persistence diagrams with different total charges will be infinitely far apart.
    This rigidity can be mitigated by increasing the values of two persistence diagrams along their diagonals so that they have the same total charges.
\end{enumerate}
These modifications lead to the following new definitions.

\begin{defn}\label{def: metric-lattice}
    A \define{finite (extended) metric lattice} is a pair $(P, d_P)$ where $P$ is a finite lattice and $d_P: P^{\op} \times P \to \mathbb R^{\geq 0} \cup \{\infty\}$ is both a metric on $P$ and an order-preserving function on $P^{\op} \times P$. 
\end{defn}

\begin{defn}[Modifies Definition $6.5$]
    The \define{category of integral functions}, denoted $\Fnc$, is the category whose objects are functions $\partial f : \bar P \to \Z$ where $P$ is a finite metric lattice and $f: \bar P \to \Z$ is a monotone function.
    Morphisms from $\partial f: \bar P \to \Z$ to $\partial g: \bar Q \to \Z$ are given by bounded lattice maps $\bar \alpha: \bar P \to \bar Q$ such that for any $I \in \bar Q$,
    \begin{equation}\label{eq: push-forward}
        \partial g(I) = \sum_{J \in \bar\alpha^{-1}(I)} \partial f (J).
    \end{equation}
\end{defn}

With these modifications, we show that the edit distance in the category of integral functions is non-trivial by proving a lower bound theorem, Theorem~\ref{thm:lower-bound}.
To state this theorem, we first need the following definitions.
\begin{defn}
    Let $(P, d_P)$ be a finite metric lattice.
    An up-set of $P$ is a subposet $A \subseteq P$ such that if $a \in A$ and $b \in P$ with $a \leq b$ then $b \in A$.
    Given a monotonic function $f: \bar P \to \Z$, we are particularly interested in up-sets of the form $f^{\geq i} := \{ I \in \bar P \mid f(I) \geq i \} \subseteq \bar P$ for $i \in \Z$.
\end{defn}

Every upset $A$ can be uniquely characterized by its set of minimal elements
\[
    \min (A) := \{ a \in A \mid \text{if } \exists a' \in A \text{ with } a' \leq a \text{ then } a' = a \}.
\]

\begin{defn}
    Let $(P, d_P)$ be a finite metric lattice and let $A \subseteq \bar P$ be an up-set.
    The \define{birth diameter} of $A$ is
    \[
        \diam_b(A) := \max_{[a,b] \in A} d_P(a, \top_P).
    \]
    Similarly, the \define{death diameter} of $A$ is
    \[
        \diam_d(A) := \max_{[a,b] \in A} d_P(b, \top_P).
    \]
    If $A$ is empty then we set $\diam_b (A) = \diam_d(A) = 0$.
\end{defn}

The order-preserving assumption on metric lattices introduced in Definition~\ref{def: metric-lattice} implies that both the birth diameter and the death diameter of an up-set $A$ will always be attained by minimal elements of $A$.

\begin{thm}\label{thm:lower-bound}
    Let $\partial f: \bar P \to \Z$ and $\partial g: \bar Q \to \Z$ be objects in $\Fnc$.
    Define ${D_f : \Z \to \mathbb R^2}$ as the function 
    \[
        D_f(i) =  \big( \diam_b (f^{\geq i}), \diam_d (f^{\geq i}) \big)
    \]
    and define $D_g : \Z \to \mathbb{R}^2$ analogously.
    Then
    \begin{equation}\label{eq: lb}
        d_\Fnc(\partial f, \partial g) \geq || D_f (i) - D_g (i) ||_\infty
    \end{equation}
    for all $i \in \Z$.
\end{thm}

Consider the persistence diagrams $\sigma_1$ and $\sigma_3$ presented in Example~\ref{ex:zero-distance}.
Both $\sigma_1$ and~$\sigma_3$ are the M\"obius inversions of monotone integral functions $f_1$ and $f_3$ respectively.
We have that $D_{f_1} (1) = (1, 0)$ and $D_{f_3} (1) = (7, 0)$ so Theorem~\ref{thm:lower-bound} guarantees that $d_\Fnc(\sigma_1, \sigma_3) \geq 6$.
Note that the problematic persistence diagram, $\sigma_2$, is no longer in $\Fnc$ as it is not the M\"obius inversion of a monotone function.

\begin{ex}
Let $[0,10]$ be the totally ordered poset of natural numbers from $0$ to~$10$, and let
$P := [0,10]^2$ be the product poset.
Define $d_P \big( (a,b) , (c,d) \big)$ as $\max \left\{ |a-c|, |b-d| \right\}$. 
Note that $d_P$ is order-preserving.
Consider the two objects $\partial f, \partial g : \bar P \to \Z$ of $\Fnc$ illustrated in Figure~\ref{fig:erratum_examples}.
The function $\partial f$ assigns to the interval starting at $(0, 0)$ and ending at $(10, 10)$ (that is, $\big[ (0,0), (10,10) \big] \in \bar P$) the value~$1$ and, to the rest of $\bar P$, the value $0$.
The function $\partial g$ assigns to $\big[ (0,5), (10,10) \big]$ and $\big[ (5,0), (10,10) \big]$ the value~$1$, to $\big[ (5,5), (10,10) \big]$ the value~$-1$, and the rest~$0$.
We now compute the lower bound of Theorem~\ref{thm:lower-bound}.
The upsets $f^{\geq i}$ and $g^{\geq i}$ are generated by the following sets of minimal intervals:
\[
    \setlength\arraycolsep{10pt} 
    \renewcommand\arraystretch{1.5} 
    \begin{array}{|c|c|c|}
    \hline
    i & \min \big( f^{\geq i} \big) & \min \big( g^{\geq i} \big) \\ \hline 
    0 &  \left\{ \big[(0,0), (0,0) \big] \right\} & \left\{ \big[(0,0), (0,0) \big] \right\} \\ \hline
    1 & \left\{ \big[(0,0), (10,10) \big] \right\} & \left\{ \big[(0,5), (10,10) \big], \big[(5,0), (10,10) \big] \right\} \\ \hline
    \end{array}
\]
The lower bound appearing in Equation~\ref{eq: lb} is calculated as follows:
\[
    \setlength\arraycolsep{10pt} 
    \renewcommand\arraystretch{1.5} 
    \begin{array}{|c|c|c|c|}
    \hline
    i & D_f(i) & D_g(i) & \Vert D_f(i) - D_g(i) \Vert_\infty \\ \hline
    0 & (10, 10) & (10, 10) & 0 \\ \hline
    1 & (10, 0) & (5, 0) & 5 \\ \hline
    \end{array}
\]
Thus, $d_{\Fnc} (\partial f, \partial g) \geq 5$.

\begin{figure}
\centering
\begin{subfigure}{0.4\textwidth}
	\centering
    \includegraphics[scale = 0.9]{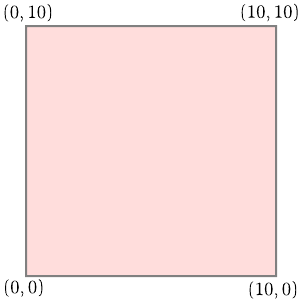}
    \caption{$\partial f : \bar P \to \Z$}
    \label{fig:first}
\end{subfigure}
\hfill
\begin{subfigure}{0.4\textwidth}
	\centering
    \includegraphics[scale = 0.9]{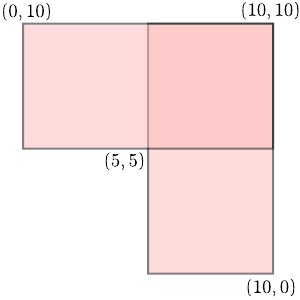}
    \caption{$\partial g : \bar P \to \Z$}
    \label{fig:second}
\end{subfigure}
        \caption{Two objects in $\Fnc$ where $P$ is the subposet $[0,10]^2 \subseteq \Z \times \Z$ with the product ordering.}
\label{fig:erratum_examples}
\end{figure}

\end{ex}

To prove Theorem~\ref{thm:lower-bound}, we'll use the following lemmas.

\begin{lem}[Rota's Galois Connection Theorem; \cite{GaloisConnections}]\label{lem:RGCT}
    Let $\bar \alpha: \bar P \to \bar Q$ be a bounded lattice map and let $\bar \beta: \bar Q \to \bar P$ be the map sending $I \in \bar Q$ to $I^* := \max \bar\alpha^{-1}[\bot_{\bar Q}, I]$.
    Then for any $f: \bar P \to \Z$ in $\Mon$ and any $I \in \bar Q$,
    \[
        \big( \partial ( f \circ \bar\beta) \big)(I) = \sum_{J \in \bar \alpha^{-1}(I)} \partial f(J).
    \]
\end{lem}

\begin{lem}\label{lem:mon-and-fnc}
    Let $\partial f: \bar P \to \Z$ and $\partial g: \bar Q \to \Z$ be objects in $\Fnc$.
    Then
    \[
        d_\Fnc(\partial f, \partial g) = d_\Mon (f,g).
    \]
\end{lem}
\begin{proof}
    We will prove this by showing that a bounded lattice map $\bar\alpha: \bar P \to \bar Q$ induces a morphism from $\partial f$ to $\partial g$ in $\Fnc$ if and only if $\bar \alpha$ induces a morphism from $f$ to $g$ in $\Mon$.
    Let~$\bar \beta: \bar Q \to \bar P$ be the map sending $I$ to $I^*$ as in Lemma~\ref{lem:RGCT}.
    We have that $\bar \alpha$ is a morphism in $\Mon$ if and only if $g(I) = f(I^*)$ for all $I \in \bar Q$.
    This is equivalent to $g = f \circ \bar\beta$.
    Because the M\"obius inversion operator is invertible, this is equivalent to $\partial g = \partial (f \circ \bar \beta)$ which is true if and only if $\partial g(I) = \big(\partial (f \circ \bar \beta)\big)(I)$ for all $I \in \bar Q$.
    Now by Lemma~\ref{lem:RGCT},
    \[
        \partial g(I) = \big(\partial (f \circ \bar \beta)\big)(I) = \sum_{J \in \bar \alpha^{-1}(I)} \partial f(J)
    \]
    for all $I \in \bar Q$.
    This is equivalent to $\bar \alpha$ being a morphism from $\partial f$ to $\partial g$ in $\Fnc$.
    This implies that $d_\Fnc(\partial f, \partial g) = d_\Mon (f,g)$ as desired.
\end{proof}

Now the problem of showing that $d_\Fnc$ is non-trivial has been reduced to showing that~$d_\Mon$ is non-trivial.
This will be established in Lemma~\ref{lem:dmon-lb} using Lemma~\ref{lem:supset-containment}.

\begin{lem}\label{lem:supset-containment}
    Let $f: \bar P \to \Z$ and $g: \bar Q \to \Z$ be objects in $\Mon$ and let $\bar \alpha$ be a morphism from $f$ to $g$.
    Then $\bar \alpha (f^{\geq i}) \subseteq g^{\geq i}$ for any $i \in \Z$.
\end{lem}
\begin{proof}
    Let $i \in \Z$ and $I \in f^{\geq i}$.
    Observe that $\bar \alpha(I)^* := \max \bar \alpha^{-1}[\bot_Q, \bar \alpha(I)] \geq I$ because $I \in \bar \alpha^{-1}[\bot_Q, \bar \alpha(I)]$.
    Now since $\bar \alpha$ is a morphism, $g \big( \bar \alpha(I) \big) = f \big( \bar\alpha(I)^* \big)$.
    Finally, because $f$ is monotone, we have $f \big( \bar\alpha(I)^* \big) \geq f(I) \geq i$ and therefore $\bar \alpha (I) \in g^{\geq i}$.
\end{proof}

\begin{lem}\label{lem:dmon-lb}
    Let $f: \bar P \to \Z$ and $g: \bar Q \to \Z$ be objects in $\Mon$ and let $ \bar \alpha: P \to Q$ be a morphism from $f$ to $g$.
    Then,
    \[
        || \bar \alpha || \geq ||D_f (i) - D_g(i)||_\infty = \max \Big\{ \big| \diam_b (f^{\geq i}) - \diam_b (g^{\geq i}) \big|,\, \big| \diam_d (f^{\geq i}) - \diam_d (g^{\geq i}) \big| \Big\}
    \]
    for any $i \in \Z$.
\end{lem}
\begin{proof}
    Because $\bar \alpha$ is a morphism from $f$ to $g$, $f(\top_{\bar P}) = g(\top_{\bar Q})$ so if $i > f(\top_{\bar P})$ then both $f^{\geq i}$ and $g^{\geq i}$ are empty.
    Therefore, choose $i \leq f(\top_{\bar P})$.
    We will first show that
    \[
        ||\bar \alpha|| \geq |\diam_b (f^{\geq i}) - \diam_b(g^{\geq i})|
    \]
    by considering the two cases where ${\diam_b (f^{\geq i}) \geq \diam_b(g^{\geq i})}$ and ${\diam_b (f^{\geq i}) < \diam_b(g^{\geq i})}$.

    If $\diam_b(f^{\geq i}) \geq \diam_b(g^{\geq i})$ then choose ${[a,b] \in f^{\geq i}}$ with $d_P (a, \top_P) = \diam_b(f^{\geq i})$.
    By Lemma \ref{lem:supset-containment}, $\bar \alpha([a,b]) = [\alpha(a), \alpha(b)] \in g^{\geq i}$ and so $d_Q (\alpha(a), \top_{Q}) \leq \diam_b(g^{\geq i})$.
    Now
    \begin{align*}
        \diam_b(f^{\geq i}) - \diam_b(g^{\geq i}) &= d_{P} (a, \top_{P}) - \diam_b(g^{\geq i})\\
        &\leq d_{P} (a, \top_{P}) - d_{Q} (\alpha(a), \top_{Q})\\
        &\leq ||\alpha|| = ||\bar \alpha||.
    \end{align*}

    Now suppose $\diam_b(f^{\geq i}) < \diam_b(g^{\geq i})$.
    Then choose a minimal $J = [c,d] \in g^{\geq i}$ with $d_{Q}(c, \top_{Q}) = \diam_b(g^{\geq i})$.
    From the assumption that $d_{Q}$ is order-preserving, we can always choose such a~$J$.
    Let $[c^*, d^*] = J^* := \max \bar \alpha^{-1}[\bot_{\bar Q}, J]$.
    Because $\alpha$ is a morphism, $f(J^*) = g(J) \geq i$ so $J^* \in f^{\geq i}$ and, by Lemma \ref{lem:supset-containment}, we have $\bar \alpha (J^*) \in g^{\geq i}$.
    Since $\bar \alpha(J^*) := \bar \alpha \big( \max \bar\alpha^{-1} [\bot_{\bar Q}, J] \big) \leq J$, the minimality of $J$ in $g^{\geq i}$ now implies that $\bar \alpha(J^*) = J$ or, equivalently, $\alpha(c^*) = c$ and $\alpha(d^*) = d$.
    This gives
    \begin{align*}
        \diam_b(g^{\geq i}) - \diam_b(f^{\geq i}) &= d_{Q}(c, \top_{Q}) - \diam_b(f^{\geq i})\\
        &= d_{Q}(\alpha(c^*), \top_{Q}) - \diam_b(f^{\geq i})\\
        &\leq d_{Q}(\alpha(c^*), \top_{Q}) - d_{P} (c^*, \top_{P})\\
        &\leq ||\alpha|| = ||\bar \alpha||.
    \end{align*}

    The same argument applied to the death coordinates of intervals gives
    \[
        ||\bar \alpha|| \geq |\diam_d (f^{\geq i}) - \diam_d(g^{\geq i})|.
    \]
\end{proof}

Theorem~\ref{thm:lower-bound} now follows from Lemma~\ref{lem:dmon-lb} and Lemma~\ref{lem:mon-and-fnc} by choosing two integral functions $\partial f_0$ and $\partial f_n$ in $\Fnc$ and considering a path $\mathcal P$
\begin{center}
    \begin{tikzcd}
        \partial f_0\ar[r, leftrightarrow, "\bar\alpha_0"]    &\partial f_1\ar[r, leftrightarrow, "\bar\alpha_1"] &\cdots\ar[r, leftrightarrow, "\bar\alpha_{n-2}"]   &\partial f_{n-1}\ar[r,leftrightarrow, "\bar\alpha_{n-1}"]    &\partial f_n.
    \end{tikzcd}
\end{center}
Focusing solely on birth coordinates for the moment, Theorem~\ref{thm:lower-bound} implies
\[
    \text{length}(\mathcal P) = \sum_{i=0}^{n-1} ||\bar \alpha_i|| \geq \sum_{i=0}^{n-1} \Big| \diam_b(f_i^{\geq i}) - \diam_b(f_{i+1}^{\geq i}) \Big| \geq \Big| \diam_b(f_0^{\geq i}) - \diam_b(f_n^{\geq i}) \Big|
\]
for any $i \in \Z$.
This argument applied to death coordinates implies the analogous lower bound and completes the proof of Theorem~\ref{thm:lower-bound}.

\newpage

\bibliographystyle{plain}
\bibliography{ref}{}


\end{document}